\documentclass[reqno,centertags,a4paper,12pt]{amsart}
\usepackage{amssymb}
\usepackage{amsthm}
\usepackage{eucal}
\usepackage{color}
\usepackage[english]{babel}
\usepackage{graphicx}
\usepackage{subfig}

\newcommand{\R}{\mathbb R}

\newtheorem{theorem}{Theorem}[section]
\newtheorem{proposition}[theorem]{Proposition}
\newtheorem{remark}[theorem]{Remark}
\newtheorem{lemma}[theorem]{Lemma}
\newtheorem{corollary}[theorem]{Corollary}

\begin{document}
\title[Higher-order Hamiltonian Model]{Higher-order Hamiltonian Model for Unidirectional Water Waves}
\author{J. L. Bona}
\address{Department of Mathematics, Statistics and Computer Science, The University of Illinois at Chicago, 851 S. Morgan Street MC 249, 
Chicago, IL 60607 USA}
\email{jbona@uic.edu}
\author{X. Carvajal}
\address{Instituto de Matem\'atica, UFRJ, 21941-909, Rio de Janeiro, RJ, Brazil}
\email{carvajal@im.ufrj.br}
\author{M. Panthee}
\address{IMECC-UNICAMP\\
13083-859, Campinas, S\~ao Paulo, SP,  Brazil}
\email{mpanthee@ime.unicamp.br}
\author{M. Scialom}
\address{IMECC-UNICAMP\\
13083-859, Campinas, S\~ao Paulo,  SP, Brazil}
\email{scialom@ime.unicamp.br}

\keywords{Nonlinear dispersive wave equations, Water wave models, KdV equation, BBM equation, Cauchy problems, local \& global  well-posedness}
\subjclass[2010]{35A01, 35Q53}
\begin{abstract}
Formally second-order correct, mathematical descriptions of long-crested water waves propagating mainly in one direction are derived.  These equations are analogous to the first-order approximations of KdV- or BBM-type.  The advantage of these more complex equations  is that their solutions corresponding to physically relevant initial perturbations of the rest state may be accurate on a much longer time scale.  The initial-value problem for the class of equations that emerges from our derivation is then considered.  A local well-posedness theory is straightforwardly established by way of a contraction mapping argument.  A subclass of these equations possess a special Hamiltonian structure that implies the local theory can be continued indefinitely.  
\end{abstract}

\maketitle

\section{Introduction}
Long-crested water waves propagating shoreward are commonplace in the shallow water zone of large bodies of water.  Waves of this general form are easily generated in laboratory settings as well.  If  
a standard $xyz$--coordinate system is adopted in which $z$ increases in the direction opposite to which gravity acts, such waves are often taken to propagate 
along the $x$--axis, say in the direction of increasing values, and to be independent of the $y$-coordinate.   In this case, if dissipation and surface tension effects are ignored, the fluid assumed to be incompressible and the motion irrotational, the standard representation of the velocity field and  the free surface is provided by the Euler equations for the motion of a perfect fluid with the boundary behavior at the free surface determined by the Bernoulli condition.  On typical geophysical length scales, these equations provide reasonably good approximations of what is actually observed in nature.  In detail, this system has the form
\begin{equation} \label{euler}
\begin{cases}
 \Delta{\varphi} = 0,\qquad\qquad & 0 <  y < h_0 + \eta(x,t), \\
  \partial_y\varphi = 0, & y = 0,
 \\
 \partial_t  \eta =  \partial_y \varphi
  - \partial_x{\eta} \cdot \partial_x{\varphi}, & y = h_0 + \eta(x,t), 
   \\
‎ \partial_t \varphi = g\eta  - \frac{1}{2} (\partial_x\varphi)^2 - \frac{1}{2} (\partial_y \varphi)^2,
 & y = h_0 +  \eta(x,t). 
 \end{cases}
\end{equation}
Here, the bottom is taken to be flat, horizontal and located at $z = 0$, though theory
with a slowly varying bottom can easily be derived along the same lines (see \cite{min}). 
The undisturbed depth is $h_0$    
while the
dependent variable, $\eta = \eta(x, t)$ is the deviation of the free surface from its rest position 
$(x,h_0)$ at time $ t$.  
 Thus,  the  depth of the water column over
the spatial point $(x, 0)$ on the bottom, at time $t$ is $h(x, t) = h_0+\eta(x, t)$.
The dependent variable
 $\phi = \phi(x, y, t)$ is the velocity potential which is defined throughout the 
 flow domain, and whose existence owes to the
 fact that the fluid is  incompressible and irrotational.  
Hence,  $(u(x, z, t), v(x, z, t)) = \nabla\phi(x, z, t)$ is the velocity field at the point $(x, z)$ in the  flow domain at time $t$.  Here, $\nabla$ connotes the gradient 
with respect to the spatial variables only.   Of course, for this formulation to make sense, it must be the case that the free surface remains a graph over the bottom, a presumption that overlies the developments here. It deserves remark that the system (\ref{euler}) can be rewritten in a Hamiltonian form, as Zakharov \cite{Z} pointed out almost 50 years ago.
\\

Beginning already in the first half of the $19^{th}$--century, simpler  models have been posited, in part because the approximation using (\ref{euler}) is both analytically and computationally recalcitrant. Note in particular that the location of the free surface is part of the problem, so that two
boundary conditions at the free surface are needed for its determination. Observe also that
the temporal derivatives only appear in the boundary conditions, making the problem
further non-standard.  Moreover, the precision one might hope for from using the Euler equations is not always needed in practice.  If the input data has significant error, there may be little point in the higher accuracy afforded by the Euler system (\ref{euler}) as opposed to cruder approximations. 

 The largest steps forward in the $19^{th}$ century study of  approximate models were taken by Boussinesq in the 1870's (see especially his opus \cite{boussinesq}).   The coupled systems of equations which now bear his name are well known to theoreticians and they and their relatives find frequent use in practical situations (see, e.g. \cite{boczar}, \cite{min}).  In addition to the presumption that the wave motion is long-crested, so sensibly one-dimensional, they subsist on the assumption that the wave amplitudes and wavelengths encountered in the evolution are, respectively, small and large relative to the undisturbed depth $h_0$ of the liquid over the horizontal, featureless bottom. 
More precisely, their derivation needs that 
\begin{equation}\label{alpha-1}
\alpha = \frac{A}{h_0}\ll 1,  \qquad \beta = \frac{h_0^2}{l^2}\ll 1, \qquad S = \frac{\alpha}{\beta} = \frac{Al^2}{h_0^3}\approx 1.
\end{equation} 
Here, $A$ is a typical amplitude of the wave motion in question while $l$ is a 
typical wavelength.  The assumption that the Stokes' number 
$S = \frac{\alpha}{\beta}$ is of order one 
effectively means that nonlinear  and dispersive effects are balanced.
 Boussinesq also derived a model, now called the Korteweg--de Vries (KdV) equation, which was a specialization of the coupled systems, formally valid for waves traveling only in one direction, say in the direction of increasing values of $x$.  

 Almost a century later, Peregrine \cite{peregrine} and Benjamin {\it et al.} \cite{BBM} returned to Boussinesq's unidirectional model
\begin{equation}  \label{kdv}
\eta_t + \eta_x + \frac32 \eta\eta_x + \frac16 \eta_{xxx} = 0
\end{equation}
(the Korteweg-de Vries equation, commonly referred to as the KdV equation)  
and derived an equivalent version  known as the regularized long-wave equation (RLW-equation) or the BBM-equation.  In terms of the dependent variable $\eta(x,t)$, this equation takes the form
\begin{equation}\label{BBM}
\eta_t + \eta_x + \frac32 \eta\eta_x - \frac16 \eta_{xxt} = 0
\end{equation}
in the unscaled, non-dimensional variables
$$
x = \frac{1}{h_0} \, \bar{x}, \quad t = \sqrt{\frac{g}{h_0}} \, \bar{t} \quad {\rm and} \quad \eta = \frac{1}{h_0}\, \bar{\eta}.
$$
Here,  the constant $g$ is the acceleration due to gravity while $\bar{x}, \bar{t}$ and $\bar{\eta}$ are laboratory or field variables, all  measured in the unit of length consistent with 
the values of $h_0$ and $g$.  

Models like the BBM- and KdV-equation are known to provide good approximations of 
unidirectional solutions of the full 
water wave problem
 (\ref{euler}) on the so-called Boussinesq time scale, 
$\frac{1}{\beta} \approx \frac{1}{\alpha}$  (see \cite{AABCW}, \cite{BCL1}, \cite{BPS1}). They are also known to predict laboratory observations with reasonable accuracy on similar time scales (see \cite{BPS}, \cite{hammack}, \cite{HS}). 

 In some applications, notably coastal engineering and ocean wave modeling,  the waves need to be followed on time scales longer than the  Boussinesq time scale  (for example, see \cite{boczar} and references therein).  In such situations, a higher-order approximation to the water-wave problem might prove to be useful as it would be formally valid
 on the square $\frac{1}{\beta^2} \approx \frac{1}{\alpha^2} $ of
 the long, Boussinesq time scale.  Such models have appeared in the literature before (see  \cite{olver1, olver} for early examples).  
It is our purpose here to put forward a class of such higher-order correct, unidirectional evolution equations and to provide  analysis relating to the fundamental issue of Hadamard well-posedness for a subclass.   Models will be isolated
 that are not only a formally second-order correct approximation of the full, two-dimensional water wave problem, but also
possesses a Hamiltonian structure. As P. Olver pointed out in his pioneering work \cite{olver}, this helpful aspect is more difficult to attain in higher-order models 
that formally are faithful to the overlying Euler equations
 than in the first-order correct KdV or BBM models. Indeed, the fifth-order model appearing in \cite{olver} does not in fact have a Hamiltonian structure, as Olver 
points out.  

The notion of well-posedness which is featured here was
 put forward by Hadamard more than a century ago in a lecture the 
well known French mathematician gave at Princeton University 
(see \cite{hadamard}). In his conception, a problem is well-posed subject to given auxiliary data when there corresponds a unique solution which depends continuously on variations in the specified supplementary data. Hadamard points out that if the problem is lacking these properties, it will probably be useless in practical applications. Auxiliary data brought from real world situations typically features at least a small amount of error. If the model were to respond discontinuously
 to these small perturbations, the reproducibility of the model predictions
 in laboratory and field settings would be compromised and likewise their 
use in real situations would be suspect.

To clarify the role of the size restrictions (\ref{alpha-1}), it is often
 helpful to rescale the 
variables.  For example, in the context of equation (\ref{BBM}), 
change variables by letting $\eta \hookrightarrow \alpha \eta$, and $(x,t) \hookrightarrow \sqrt{\beta}(x,t)$.   In the 
new variables, $\eta$ and its first few partial derivatives with respect to $x$ and $t$ are presumed to be of order one and the equation takes 
the form
\begin{equation} \label{scaledbbm}
\eta_t + \eta_x + \frac32 \alpha \eta\eta_x - \frac16 \beta \eta_{xxt} = 0.
\end{equation}
In this scaling, the role of the small parameters is more apparent.  Moreover, 
the error term made in the approximation, which is set to zero in (\ref{scaledbbm}), 
is quadratic in the small parameters $\alpha$ and $\beta$.  Because of this latter 
aspect, even though the solution and its derivatives remains of order one, the ignored error can accumulate and have an order-one effect on the 
solution on a time scale of size $\frac{1}{\alpha^2} \approx \frac{1}{\beta^2}$, 
hence the need for a higher-order correct model if longer spatial distances are 
in question.

The starting point of our derivation of higher-order KdV-BBM-type 
equations is the paper \cite{BCS1} 
(and see also the earlier note \cite{min})  where a several-parameter variant of the classical Boussinesq system of two coupled equations was derived.  These Boussinesq systems are derived without the assumption of one-way propagation and can therefore countenance long-crested waves propagating in both directions. 
The theory in \cite{BCS1} assumes  incompressibility, irrotationality, long-crestedness and 
the size conditions enunciated in (\ref{alpha-1}). Boussinesq systems were formally derived at both first and second order in the small parameters $\alpha$ and 
$\beta$.   In dimensionless, scaled variables 
as appearing in (\ref{scaledbbm}), 
the family of formally first-order correct systems has  the form
\begin{equation}\label{b1.1}
\begin{cases}
\eta_t +w_x +\alpha(w\eta)_x + \beta\big(aw_{xxx}-b\eta_{xxt}\big)=0,\\
w_t +\eta_x +\alpha ww_x +\beta\big(c\eta_{xxx} -dw_{xxt}\big) =0.
\end{cases}
\end{equation}
 The variable $\eta$ is proportional to the deviation of the free surface from 
its rest position at the point $x$ at time $t$, as it was in (\ref{BBM}), while $w = w(x,t)$ is proportional to the horizontal velocity at a certain depth $z_0$, say, at  the point $(x,z_0,t)$ in the flow domain.  (The velocity $w$ is scaled by $\sqrt{gh_0}$ to make it non-dimensional, and then by $\alpha$ 
to make it of order one.)  
The constants $a, b, c$ and $d$ are not arbitrary.  They satisfy the relations
\begin{equation}\label{abcd}
\begin{cases}
a=\frac12\Big(\theta^2-\frac13\Big)\lambda, \qquad & b=\frac12\Big(\theta^2-\frac13\Big)(1-\lambda),\\
c=\frac12(1-\theta^2)\mu, \qquad \quad  &d=\frac12(1-\theta^2)(1-\mu),
\end{cases}
\end{equation} 
so that $a + b + c + d = \frac13$.  
In the same, order-one, independent and dependent variables, the second-order correct systems are
\begin{equation} \label{b1.2}
\begin{cases}\begin{split}
\eta_t &+w_x  +  \beta\big(a w_{xxx} -b\eta_{xxt}\big)+\beta^2\big(a_1w_{xxxxx} + b_1\eta_{xxxxt}\big) \\
&= -\alpha (\eta w)_x + \alpha\beta\big(b(\eta w)_{xxx} -(a+b-\frac13)(\eta w_{xx})_x\big),\\
w_t&+\eta_x+\beta \big(c\eta_{xxx}-d w_{xxt}\big) + \beta^2\big(c_1\eta_{xxxxx} +  d_1w_{xxxxt} \big) \\
&=-\alpha ww_x\!   +\!\alpha \beta\big((\!c\!+\!d) ww_{xxx}\!-\!c(ww_x)_{xx}
-\!(\eta\eta_{xx})_x\! +\!(\!c\!+\!d\!-\!1\!) w_xw_{xx}\big),
\end{split}
\end{cases}
\end{equation}
where the additional constants $a_1, b_1, c_1, d_1$ are 
 \begin{equation}\label{abcd1}
 \begin{cases} \begin{split}
a_1&= -\frac14\Big(\theta^2-\frac13\Big)^2(1-\lambda)+\frac5{24}\Big(\theta^2-\frac15\Big)^2\lambda_1,\\
 b_1& = -\frac5{24}\Big(\theta^2-\frac15\Big)^2(1-\lambda_1),\\
 c_1& =\frac5{24}(1-\theta^2)\Big(\theta^2-\frac15\Big)(1-\mu_1),\\
 d_1&=-\frac14\big(1-\theta^2\big)^2\mu-\frac5{24}(1-\theta^2)\Big(\theta^2-\frac15\Big)\mu_1.
 \end{split}
\end{cases}
\end{equation}
The parameter $\theta$ has physical significance.  It is determined by the height above the bottom at which the horizontal velocity is specified initially and whose evolution is being followed.  In the earlier notation, $\theta = 1-z_0$.    Because 
the vertical variable is scaled by the undisturbed depth $h_0$ in these descriptions, $\theta$ must lie in the interval $[0,1]$.  The other values, $\lambda, \mu, \lambda_1$ and $\mu_1$ are modeling parameters and
can  in principle take any real value.  Thus the coefficients appearing in the higher-order Boussinesq systems
form a restricted, eight-parameter family.   Notice that if terms
quadratic in $\alpha$ and $\beta$ are dropped, the second-order system (\ref{b1.2}) reduces to the first-order system (\ref{b1.1}). 

The velocity field in the rest of the flow is determined by an associated approximation of the velocity potential in the flow-domain.  The latter is derived from a knowledge of $w$ (see \cite{BCS1}, \cite{BCG}).

Local in time well-posedness of the Cauchy problem for the systems  (\ref{b1.1})  and (\ref{b1.2}) was studied in \cite{BCS1} and \cite{BCS2}.  Not 
all of these systems are even linearly well-posed.  Indeed, the recent foray \cite{AB} 
shows that many of those not linearly well-posed are in fact not locally well-posed 
when the nonlinearity is taken into account.   The fact that some of the family 
is ill-posed has the advantage of eliminating  
them from consideration when real-world 
approximation is the goal.  

These systems were further extended in \cite{BCL1} to include waves that are fully three-dimensional, and not just long-crested motions.   Rigorous estimates 
were also provided for the difference between solutions of the full water-wave problem and solutions of the first-order models.   A further extension of \cite{BCL1} is given in \cite{LS1}, where 
Boussinesq systems in the Kadomtsev-Petviashvili (KP) scaling are derived.  The latter situation is intermediate between the long-crested regime 
where transverse motion is ignored entirely and three-dimensional Boussinesq systems that allow strong transverse disturbance; a regime that 
is often referred to as allowing for weakly transverse long waves.  A detailed survey of  results of this sort can be found in J.-C. Saut's lecture notes \cite{Saut} or the recent monograph of Lannes \cite{DL}.

As hinted already, when long-crested waves are essentially moving 
in only one direction, one might prefer to use a unidirectional model 
because less auxiliary data is needed 
to initiate it. Theory developed in \cite{AABCW} has shown rigorously that predictions of first-order Boussinesq systems and those of their unidirectional
counterpart (\ref{BBM}) are the same to the neglected order, provided the wave motion is initiated unidirectionally.  This gives  rigorous credence to the utility of such unidirectional  models since 
the bidirectional models are known to be a good approximation of solutions of the full Euler
system in the Boussinesq regime of small amplitude and long wavelength.    

We stress that while the higher-order, unidirectional models put forward 
here are formally correct 
on the square of the Boussinesq time scale, no proof of this exists.  Indeed, considering the difficulty encountered in showing the first-order
correct, Boussinesq systems are faithful to the full, inviscid water wave problem 
\eqref{euler} on the Boussinesq time scale and showing the KdV-BBM approximations
\eqref{kdv}-\eqref{BBM} 
is true to their overlying Boussinesq system, 
a rigorous result for the systems derived here 
 on the square of the Boussinesq time scale is likely to be
challenging.  One can show that the higher-order terms do not do damage
to the original  KdV-BBM approximation of the full water-wave problem on the Boussinesq time scale, {\it  provided} 
sufficiently smooth initial data are countenanced.  This point is not addressed here 
as it would take us afield of the main developments.  It is also the case that one can 
show directly and rigorously that the linearized, higher-order, unidirectional model 
is faithful to the linearized Boussinsq system  on this very long time scale, again, provided the initial data has 
enough regularity.   However, these results are far from what one would like to have 
in hand.

The present contribution proceeds as follows.  In the next section, we derive formally from the second-order Boussinesq equations a class of second-order 
KdV-BBM-type equations.  
Also in the next section, function class notation is introduced and our main results about the higher-order, unidirectional models are stated.  Section \ref{sec4} provides  proofs of the results stated in Section \ref{sec3}, while Section \ref{sec5} features commentary about the choice 
of the parameters $\theta, \lambda, \mu, \lambda_1$, $ \mu_1$ and another 
parameter $\rho$ to be introduced presently.  
Section \ref{sec-5} is devoted to a discussion of the linear 
  dispersion relation. Finally, in Section \ref{sec-6}  some concluding remarks are recorded.

\section{Derivation of the Models and the Main Results}\label{sec2}

The formal derivation of a class of higher-order,  unidirectional equations, together with a precise statement of results about their well-posedness is the topic of this section.
\subsection{Model Equations}
The starting point is the collection (\ref{b1.2})
of higher-order  Boussinesq systems derived in \cite{BCS1}.   The parameters 
$a, b, \cdots c_1, d_1$ are those presented in
(\ref{abcd}) and (\ref{abcd1}).    As we are working in the Boussinesq regime where the Stokes' number $S = \frac{\alpha}{\beta} \approx 1$, 
the two small parameters $\alpha$ and $\beta$ are treated on an equal footing.  Thus, $O(\alpha) = O(\beta), \, O(\alpha\beta) = O(\beta^2)$, etc. 

In case the wave motion is essentially  in one direction, say in the direction of increasing values of $x$, we will show how to reduce such Boussinesq systems to the single,  fifth-order model,
\begin{equation}\label{eq.m1}
\begin{split}
\eta_t+\eta_x- \beta\gamma_1\eta_{xxt}+& \beta\gamma_2\eta_{xxx}+\beta^2 \delta_1 \eta_{xxxxt}+\beta^2 \delta_2 \eta_{xxxxx}+\alpha\frac34(\eta^2)_x\\
&+\,\alpha\beta\Big(\gamma(\eta^2)_{xxx}-\frac7{48}(\eta_x^2)_x \Big) -\alpha^2\frac18(\eta^3)_x=0.
\end{split}
\end{equation}
  The constants $\gamma_1$, $\gamma_2$, $\delta_1$, $\delta_2$ and $\gamma$ depend upon the parameters $a, b, \cdots$ in (\ref{b1.2}) and will be displayed presently.

Passage from the Boussinesq systems (\ref{b1.2}) to the unidirectional models (\ref{eq.m1}) follows the same line of argument as did the passage from the first-order system (\ref{b1.1}) to the mixed
KdV--BBM equations
\begin{equation} \label{kdvbbm}
\eta_t + \eta_x + \frac32 \alpha \eta \eta_x + \nu \beta \eta_{xxx} 
-\Big (\frac16 - \nu \Big)\beta \eta_{xxt} \, = \, 0,
\end{equation}
where $\nu =\frac12( a+c) = \frac14\big[\theta^2(\lambda - \mu) - \frac13 \lambda + \mu \big]$ depends upon $\theta, \lambda$ and $\mu$ and can formally take any real value.   (See \cite{AABCW}, \cite{CL} and, in the internal wave context, \cite{DV}.  A special case of this model may be 
 found in \cite{varlamov} for a moving boundary problem.)

As described in \cite{Bona}, at the lowest order of approximation wherein the parameters are small enough that even the first-order terms in  $\alpha$ and $\beta$ may be dropped, the system (\ref{b1.2}) becomes the one-dimensional wave equation, {\it viz.}
\begin{equation}\label{linear-1}
\begin{cases}
\eta_t +w_x=0,\qquad\quad
w_t+\eta_x =0,\\
\eta(x,0) =f(x), \qquad w(x,0) = g(x),
\end{cases}
\end{equation}
where $f(x)$ and $g(x)$ are the initial disturbances of the surface and the horizontal velocity, respectively. The solution to (\ref{linear-1}) is 
\begin{equation}
\begin{cases}
\begin{split}
\eta(x,t)= \frac12\big[f(x+t)+f(x-t)\big] - \frac12\big[g(x+t)-g(x-t)\big],\\ 
w(x,t)= \frac12\big[g(x+t)+g(x-t)\big] -\frac12\big[f(x+t)-f(x-t)\big].
\end{split}\end{cases}
\end{equation}
For the left-propagating component to vanish, one must have $f=g$, in which case  $\eta(x,t) =f(x-t)$,
$$
\eta_t + \eta_x = 0 \quad {\rm and} \quad w = \eta.
$$   
 Notice in particular that in the Boussinesq  regime, when most of the propagation is to the right, it 
appears that
\begin{equation}\label{dxdt}
\eta_t  =-\eta_x+O(\alpha, \beta), \quad  {\textrm{as}}\quad \alpha, \beta \to 0,
\end{equation}
a point that will play a significant role in what follows.  

At the next order when one keeps terms of first order in $\alpha$ and $\beta$, 
the standard {\it ansatz} used in \cite{AABCW} was 
that 
\begin{equation} \label{lower}
w = \eta + \alpha A + \beta B
\end{equation}
where $A = A(\eta, \cdots)$ and $B = B(\eta_{xx}, \eta_{xt}, \cdots)$ turn out to be simple polynomial functions of $\eta$ and its first few partial derivatives.  Indeed, substituting (\ref{lower}) into the first-order system (\ref{b1.1}) and dropping all terms of quadratic order in the 
small parameters $\alpha$ and $\beta$ leads to the pair
\begin{equation} \label{pair}
\begin{cases}
\begin{split}
\eta_t + \eta_x + \alpha A_x + \beta B_x +  \alpha (\eta^2)_x + \beta\big(a\eta_{xxx} - b \eta_{xxt}\big) \, = \, 0, \\
\eta_t + \alpha A_t + \beta B_t  + \eta_x + \alpha ww_x + \beta\big(c \eta_{xxx} - d \eta_{xxt} \big) \, = \, 0,
\end{split}
\end{cases}
\end{equation}
 of equations.  Demanding that these be consistent, and making use of the fact derived from (\ref{dxdt}) that $A_t = -A_x + O(\alpha, \beta)$ and
similarly for $B$, it is determined   that 
\begin{equation} \label{first-orderAB}
A = -\frac14 \eta^2 \qquad {\rm and} \qquad B = \frac12\Big((c-a )\eta_{xx} + (b-d)\eta_{xt}\Big).
\end{equation}
Using these relations in either of the equations in (\ref{pair}) leads to the KdV--BBM equations (\ref{kdvbbm}) with $\nu$ as advertised 
above.  

If one now again makes use  of the low-order relation (\ref{dxdt}) between $\partial_x$ and $\partial_t$, the equation (\ref{kdvbbm}) can be reduced 
further to the pure BBM-equation (\ref{scaledbbm}).  (The same equation can also be obtained by particular choices of $\theta, \lambda$ and $\mu$.)

It was shown in \cite{AABCW} that not only 
does this procedure lead formally to KdV--BBM-type equations of the form displayed in (\ref{kdvbbm}), but that if the Boussinesq system is initiated with data $(\eta_0, w_0)$ 
that satisfies (\ref{lower}), then its solution $(\eta,w)$ has $\eta$ well approximated by the solution $\eta_{BBM}$ of (\ref{scaledbbm}),  initiated 
with $\eta_0$, and the velocity $w$ that the Boussinesq system generates is shown to be well approximated by 
using the BBM-amplitude $\eta_{BBM}$ and the formula (\ref{lower}) to define a BBM-horizontal velocity $w_{BBM}$.

If a higher-order approximation is needed, then it is natural to posit the higher-order {\it ansatz}
\begin{equation}\label{eq1.3}
w = \eta +\alpha A+\beta B +\alpha\beta C+\beta^2D+\alpha^2E
\end{equation}
analogous to (\ref{lower}) (see, for example, \cite{dullin}, \cite{DL}).  The functions $A, B, C, D$  and $E$ will again turn out  to be polynomial functions of $\eta$ and its partial derivatives.  
It deserves remark that the presumption (\ref{eq1.3}) was already persued in \cite{olver1} and in subsequent publications, but the fifth-order 
partial differential equations that emerged do not have a Hamiltonian structure.  

Substituting  (\ref{eq1.3}) into the system (\ref{b1.2}) and ignoring terms that are 
at least cubic in the small
parameters $\alpha$ and $\beta$ leads to the pair of equations
\begin{equation}\label{system-s2}
\begin{cases}\begin{split}
\eta_t&=-\eta_x-\alpha A_x -\beta B_x -\alpha\beta C_x -\beta^2 D_x  -\alpha^2E_x+b\beta\eta_{xxt}-b_1\beta^2\eta_{xxxxt} -a\beta \eta_{xxx}\\
&\quad-a\alpha\beta A_{xxx}-a\beta^2B_{xxx}-a_1\beta^2\eta_{xxxxx} -(\alpha\eta^2+\alpha^2A\eta+\alpha\beta B\eta)_x \\
&\quad+ b\alpha\beta(\eta^2)_{xxx} -(a+b-\frac13)\alpha\beta(\eta \eta_{xx})_x,\\
\eta_t&=-\eta_x  -\alpha A_t-\beta B_t-\alpha\beta C_t-\beta^2 D_t-\alpha^2E_t+d\beta \eta_{xxt}+d\alpha\beta A_{xxt}+d\beta^2B_{xxt}\\
&\quad  - d_1\beta^2 \eta_{xxxxt}-c\beta\eta_{xxx}-c_1\beta^2\eta_{xxxxx} -\alpha \eta\eta_x-\alpha^2(\eta A)_x-\alpha\beta(\eta B)_x\\
 &\quad
 -c\alpha\beta(\eta\eta_x)_{xx} +(c+d)\alpha\beta \eta\eta_{xxx}-\alpha\beta(\eta\eta_{xx})_x +(c+d-1)\alpha\beta \eta_x\eta_{xx}.
\end{split}\end{cases}
\end{equation}
   Demanding that these two equations be consistent (at the first order) leads to the formulas 
(\ref{first-orderAB}) for $A$ and $B$ at order $\alpha$ and $\beta$, respectively, as one would expect.  
Our goal is to derive a fifth-order, one-way model
which, in addition to being Hamiltonian, has a linear dispersion relation
which matches that of the full water-wave system (\ref{euler}) up to and
including the order $\beta^2$ terms, so presenting an
 error which is formally of order $\beta^3$ (recall that $\alpha\approx \beta$ in 
the present development). 
The laboratory experiments reported in \cite{BPS1} make it clear that the 
error in the phase velocity dominates the overall error, at least for 
moderately sized waves.    Hence, getting the 
dispersion relation right to the
order we are working seems important.  Indeed, if the dispersion relation is 
not correct to order $\beta^2$, the model definitely is not second-order 
correct in the limit of very small values of $\alpha$ ({\it e.g.} linear theory).

It will be helpful to introduce an auxiliary parameter $\rho$, {\it viz.} 
\begin{equation}\label{B-1}
B=\frac12 (c-a+\rho)\eta_{xx} +\frac12(b-d+\rho)\eta_{xt}.
\end{equation}
Of course, at the first order, this is equivalent to the version with $\rho = 0$, but at the next order, $\rho$ can be chosen so that the resulting  second order, one-way model
has certain, desirable properties.  This will be discussed in more detail 
in Section 4.  
Of special interest will be the value
  \begin{equation}\label{correction-1}
\rho=b+d-\frac16.
\end{equation} 
This will turn out to be perspicuous, though we do not insist on it for the nonce.   

With this value of $B$, the  mixed KdV--BBM equation \eqref{kdvbbm} resulting from the first-order approximation turns out to be
\begin{equation} \label{kdvbbm-1}
\eta_t + \eta_x + \frac32 \alpha \eta \eta_x + \tilde\nu \beta \eta_{xxx} 
-\Big (\frac16 - \tilde\nu \Big)\beta \eta_{xxt} \, = \, 0,
\end{equation}
where $\tilde\nu =\frac12( a+c+\rho)$. Notice that if \eqref{correction-1} holds,
then $\tilde \nu = \frac{1}{12}$.  To  insist on the consistency of 
the two equations in \eqref{system-s2} at the second order in 
$\alpha$ and $\beta$, we use the approximation
\begin{equation} \label{kdvbbm-2}
\eta_t =-\eta_x - \frac32 \alpha \eta \eta_x -\tilde\nu \beta \eta_{xxx} 
+\Big (\frac16 - \tilde\nu \Big)\beta \eta_{xxt} +O(\alpha^2, \beta^2, \alpha\beta),  \quad  {\textrm{as}}\quad \alpha, \beta \to 0.
\end{equation}
 If one uses the forms of $A$ and $B$ given respectively in \eqref{first-orderAB} and \eqref{B-1} in the system  \eqref{system-s2} and the approximation \eqref{kdvbbm-2}, 
 there appear more terms involving order $\alpha\beta$, $\beta^2$ and $\alpha^2$. With this in mind, equating the terms of order $\alpha\beta$  in (\ref{system-s2}) leads to the equation
\begin{equation}\label{C-1}
C = \Big[\frac18(a+4b+2c-d)+\frac3{16}(a+b-c-d)+\frac38\rho\Big](\eta^2)_{xx} +\frac{13}{24}\eta\eta_{xx} +\frac{11}{48}\eta_x^2.
\end{equation}
Likewise, equating the terms containing  $\beta^2$  in (\ref{system-s2}) yields
 \begin{equation}\label{D-1}
 \begin{split}
D=&-\Big[\frac12(b_1-d_1)+\frac14(b-d+\rho)\big(a-d+\frac16\big)+\frac14 d(c-a+\rho)\Big]\eta_{xxxt}\\
&- \Big[\frac12(a_1-c_1)+\frac14(c-a+\rho)\big(a+\frac16\big)-\frac1{12}\rho\Big]\eta_{xxxx}.
\end{split}
\end{equation}
Finally, balancing the terms containing $\alpha^2$ in the system  (\ref{system-s2}), one obtains
\begin{equation}\label{E-1}
E= \frac18\eta^3.
\end{equation}

Putting the expressions for  $A$, $B$, $C$, $D$ and $ E$  in either equation in (\ref{system-s2}),
 using the relation (\ref{kdvbbm-2})  and taking note of the formula $\eta\eta_{xxx}= \frac12 (\eta^2)_{xxx}-\frac32(\eta_x^2)_x$, there appears the evolution equation
\begin{equation}\label{eq1.5}
\begin{split}
\eta_t+\eta_x& -\gamma_1\beta\eta_{xxt}+\gamma_2\beta\eta_{xxx}+\delta_1\beta^2\eta_{xxxxt}+\delta_2\beta^2\eta_{xxxxx}\\
&+
\frac32\alpha\eta\eta_x
+ \alpha\beta\Big(\gamma(\eta^2)_{xxx}
-\frac7{48}(\eta_x^2)_x\Big)-\frac18\alpha^2(\eta^3)_x=0,
\end{split}
\end{equation}
where
\begin{equation}\label{gamas}
\begin{cases}
\gamma_1=\frac12(b+d-\rho),\\
\gamma_2=\frac12(a+c+\rho),\\
\delta_1=  \frac14\big[2(b_1+d_1)-(b-d+\rho)\big(\frac16-a-d\big)-d(c-a+\rho)\big],\\
\delta_2=  \frac14\big[2(a_1+c_1) -(c-a+\rho)\big(\frac16-a\big)+\frac1{3}\rho\big],\\
\gamma  =\frac1{24}\big[5-9(b+d)+9\rho\big].
\end{cases}
\end{equation}
\vspace{.02cm}


\begin{remark}\label{rem-2.1}
  As our analysis so far has been predicated on the $abcd$-system \eqref{b1.2}, 
the relation 
 $a+b+c+d=\frac13$ has been used while calculating $C$ and $D$, and consequently the values of the parameters introduced in \eqref{gamas}.  In this 
situation, one readily obtains that $\gamma_1+\gamma_2 =\frac16$, $\gamma=\frac1{24}(5-18\gamma_1)$ and $\delta_2-\delta_1 = \frac{19}{360}-\frac16\gamma_1$ $($see \eqref{coefficient} below$ )$. 
Thus, equation \eqref{eq1.5} effectively has only two free parameters, namely
$\gamma_1$ and  $\delta_1 $. This aspect plays no particular role 
in the well-posedness theory to follow.  However, it does become important 
when the issue of insuring the system is Hamiltonian is addressed.  
 Detailed discussion of these issues 
may be found in Sections \ref{sec5} and \ref{sec-5}.

If instead, one were to relax  the relation $a+b+c+d=\frac13$ when
 computing $C$, $D$ and elsewhere, the resulting model would be
\begin{equation}\label{eq1.5-m}
\begin{split}
\eta_t+\eta_x& -\gamma_1\beta\eta_{xxt}+\gamma_2\beta\eta_{xxx}+\delta_1\beta^2\eta_{xxxxt}+\delta_2\beta^2\eta_{xxxxx}\\
&+
\frac32\alpha\eta\eta_x
+ \alpha\beta\Big(\sigma_1(\eta^2)_{xxx}
-\sigma_2(\eta_x^2)_x\Big)-\frac18\alpha^2(\eta^3)_x=0,
\end{split}
\end{equation}
where $\gamma_1$, $\gamma_2$ are as in \eqref{gamas},  $\delta_1$, $\delta_2$ satisfy the relation 
\begin{equation}\label{del2-del1}
\delta_2-\delta_1 = \frac14\rho(a+b+c+d)+\frac18\big[(b-d)^2-(a-c)^2\big] +\frac12(a_1-b_1+c_1-d_1)
\end{equation}
and $\sigma_1$,  $\sigma_2$ are given by 
\begin{equation}\label{sigmas}
\begin{cases}
\sigma_1=\frac1{24}\big[4+3(a-2b+c-2d)+9\rho\big],\\
\sigma_2=\frac1{48}\big[4+9(a+b+c+d)\big].
\end{cases}
\end{equation}

Note that the more general equation \eqref{eq1.5-m} reduces to 
\eqref{eq1.5} when  $a+b+c+d=\frac13$.  An in-depth 
analysis  of the general model \eqref{eq1.5-m} could be interesting. 
Such a more general model might arise if surface tension effects 
were taken into account in the original Boussinesq system.  
Depending upon
the undisturbed depth, another small parameter may arise in 
this situation and one must 
deal with its relation to $\alpha$ and $\beta$.  What the corresponding 
second-order correct model  looks like would depend upon how these 
parameters compare to one another.  This potentially interesting project 
is not pursued here.  Our focus remains upon  the one-way model 
\eqref{eq1.5} corresponding the 
the second-order water wave  system \eqref{b1.2} 
for which  dispersion considerations
mentioned earlier 
 demand that $a+b+c+d=\frac13$.

\end{remark}
\vspace{.2cm}

While the derivation is formal, we  expect the equation \eqref{eq1.5} to have the same sort of properties that its first-order 
correct analog (\ref{scaledbbm}) does as regards approximating unidirectional solutions of the second-order Boussinesq 
system (\ref{b1.2}) and, consequently, solutions of the full water wave problem.  However,  as already mentioned, rigorous theory to this effect is 
not  available as it is at first order.    

Models like \eqref{eq1.5} have appeared in the literature before ({\it cf.} 
\cite{dullin} when the surface tension is set to $0$ and 
the wide ranging article \cite{johnson} together with the references contained in these articles).  
For example, the equation (2.19) in \cite{dullin}, in the zero surface tension regime, 
appears in our class of equations (see the discussion in Sections 4 and 5).  

It is also worth note that if $\alpha = O(\beta^{\frac12})$ instead of 
$\alpha = O(\beta)$, 
then a Camassa-Holm type equation emerges, namely
\begin{equation*}
\begin{split}
\eta_t+\eta_x& -\gamma_1\beta\eta_{xxt}+\gamma_2\beta\eta_{xxx}\\
&+
\frac32\alpha\eta\eta_x
+ \alpha\beta\Big(\gamma(\eta^2)_{xxx}
-\frac7{48}(\eta_x^2)_x\Big)-\frac18\alpha^2(\eta^3)_x=0.
\end{split}
\end{equation*}
The two higher-order, linear, dispersive terms drop off because they 
are now negligible compared to the remaining terms.
  However, as one would expect for  models where the nonlinear
effects are more dominant,
the formal temporal range of validity for this model, in terms 
of the wavelength parameter $\beta$, 
is only of order $O(\beta^{-1})$.   That is to say, the formal 
error between the model predictions
 and those of the full water wave problem is of order $O(\beta^2 t)$.   
If the two fifth-order dispersive terms are left in place, then 
higher-order nonlinear terms deserve keeping as well to 
maintain a uniform level of approximation.   On the other hand, 
insofar as the largest part of the error resides in incorrect phase speeds, 
keeping these terms could be useful in practical situations, even in this
more nonlinear situation.  After all, the experiments in \cite{BPS1} 
show that BBM-type equations maintain engineering level approximation in the
long-wave regime, even for Stokes numbers in the mid-20's, which is to
say $\alpha/\beta \approx 25$.

For the analysis that follows, the small parameters $\alpha$ and $\beta$ are not relevant.   
Reverting to non-dimensional, but unscaled variables, which are denoted surmounted with a tilde, namely 
 $\tilde{\eta}(\tilde{x},\tilde{t}) = \alpha^{-1} \eta(\beta^{\frac12}\tilde{x},\beta^{\frac12}\tilde{t})$ 
and then
dropping the tildes yields the fifth-order, KdV--BBM-type equation
\begin{equation}\label{eq1.6}
\eta_t+\eta_x - \gamma_1\eta_{xxt}+\gamma_2\eta_{xxx}+\delta_1\eta_{xxxxt}+\delta_2\eta_{xxxxx}+\frac34(\eta^2)_x+\gamma (\eta^2)_{xxx}-\frac7{48}(\eta_x^2)_x -\frac18(\eta^3)_x=0.
\end{equation}
In many circumstances, boundary-value problems may be the most practically interesting.  However, one usually starts with  the pure initial-value problem to
get an idea of what may be true for more complicated problems.  This latter problem, wherein we search for a solution of (\ref{eq1.6}) subject 
to  $\eta(x,0)$ being specified for all $x \in \R$ will be the subject of further mathematical 
consideration.

We conclude this sub-section with ​the​ observation that  approximate models
 like the one displayed ​in \eqref{eq1.6} can be derived by expanding the 
Dirichlet-Neumann operator in the
Zakharov-Craig-Sulem formulation ​(​see​,​ for example​,​ \cite{DL} and the 
references therein​)​. 
An approach ​using the Dirichlet to Neumann operator does have as a component the 
rigorous theory pertaining to this operator.  And if one is expanding the Hamiltonian 
rather than the dependent variables themselves, one is guaranteed a Hamiltonian equation. However, it does not guarantee that the 
dispersion relation so obtained fits the full dispersion to the order of the terms being kept. Nor  does it  guarantee  that the resulting equation 
provides a well-posed problem.   A good 
example of what can go wrong appears in  \cite{ABN1} and   \cite{ABN2},  where this technique was applied to a deep-water situation.   Similar problems arise for
the Kaup-Boussinesq system, which is formally Hamiltonian, but is 
ill-posed even in smooth function classes (see \cite{AB}).

​The classical expansion used here ​allows ​for choices of​ parameters  that 
​guarantees both well-posedness and, in a special case, Hamiltonian structure.​ 
 It also has the advantage of
producing a model that behaves well with respect to the imposition of non-trivial boundary 
conditions (see \cite{hongqiu}).  ​

\subsection{Mathematical Theory} \label{sec3}


The equations (\ref{eq1.6}) above formally describes the propagation of uni-directional waves.   Naturally, one would like to have a theory that shows solutions of this system closely track associated solutions of the higher-order Boussinesq systems (\ref{b1.2}) on the
longer time scale $O\left(\frac{1}{\beta^2}\right)$.   Logically prior to such a result is the fundamental issue of the well-posedness of the Cauchy problem associated to (\ref{eq1.6}). It is to this latter issue that attention is now turned. 
To be useful in comparing the unidirectional model with its overlying bi-directional 
analog, one naturally needs  a well-posedness theory that is valid at least on the longer time scale $O(\frac1{\beta^2})$. Better still would be a global well-posedness theory so that issues of finite time singularity formation 
do not intrude upon the practical use of such models.

As mentioned earlier, the notion of well-posedness  used is the standard one.  
We say that the Cauchy problem for  an equation is locally well-posed in a 
Banach space $X$ of functions of 
the spatial variable if corresponding to given initial data in $X $  there exists a
 non-trivial time interval $ [0, T] $ and a unique continuous curve in $X$,
 defined at least for $t \in[0, T]$ that solves the equation in an appropriate sense.  
It is also demanded that 
this solution varies continuously with variations of the initial data. If the above properties are true for any bounded time interval, we say that the Cauchy problem is 
globally well-posed in $X$.

For the local well-posedness theory, it is important that the coefficients $\gamma_1$ and $\delta_1$ appearing respectively in front of the 
$\eta_{xxt}$ and $\eta_{xxxxt}$--terms be non-negative.
The problem is linearly ill-posed if this is not the case, as one can see by taking  the linear part of equation \eqref{eq1.8} in the next section.  
 (The special cases where $\delta_1 = 0$  is also locally 
well-posed, but will not be considered here.)
 It will be presumed henceforth that $\gamma_1 \geq 0$ and $ \delta_1 > 0$  to be the case.
Discussion of concrete conditions for this to be the case are forthcoming in Section \ref{sec5}.  Indeed, it will be shown that there are plenty of choices of the fundamental parameters $\theta, \lambda$, $\mu$, $\lambda_1$, $\mu_1$ and $\rho$    for which  $\gamma_1$, $ \delta_1$ are positive.

Local well-posedness will be obtained by using multilinear estimates combined with a contraction mapping argument. The local theory does not depend upon special choices of the parameters in the problem other than the positivity of $\gamma_1$
and $\delta_1$. In general, equation (\ref{eq1.6}) does not 
have an obvious Hamiltonian structure. However, by suitably 
choosing the parameters, it can be put into Hamiltonian form. 
The Hamiltonian structure allows one to infer bounds on solutions that
lead to global well-posedness.  As seen in the recent simulations of solutions of 
some of the first-order systems \cite{min1}, lack of Hamiltonian structure often 
seems to go along with lack of global well-posedness.  

While solutions of the system (\ref{eq1.6})  will not approximate solutions of the full water wave problem (\ref{euler}) without considerable smoothness (see \cite{BCL1}), a modern thrust in the analysis of dispersive partial differential equations is to provide local and global well-posedness theory in relatively large function classes. While mostly of mathematical interest, theory in such
low-regularity classes can be useful in the analysis of numerical schemes for approximating solutions of such equations, especially when the lower-order 
norms can be given time-independent bounds.

To obtain a global well-posedness result for initial data with lower order Sobolev regularity, we use a high-low frequency splitting technique.  Such splitting methods  have roots at least as far back as the work of M. Schonbek 
and her collaborators (see  \cite{schonbek1},   
\cite{schonbek2} for example).  In the context of BBM-type equations, it was applied in 
 \cite{BC} and \cite{BT} to obtain sharper well-posedness results.  More subtle splitting 
appears in the work of Bourgain (see, {\it e.g.} \cite{B} and the references therein, as well 
as the further developments in 
 \cite{FLP1}, \cite{FLP2} for example.)


Before announcing the main results, the mostly standard 
 notation that will be used throughout is recorded.  If $f$ is a function 
defined on the real line $\R$, then $\hat{f}$ denotes  its Fourier transform, 
namely
$$\hat{f}(\xi) = \frac1{\sqrt{2\pi}}\int_{\R}e^{-ix\xi}f(x)dx.$$  The space of square-integrable, measurable functions defined on a
measurable subset $\Omega$ of Euclidean space will be denoted $L^2(\Omega)$.  In fact, throughout, $\Omega$ will always 
be an interval in the real line $\R$ or a Cartesian product of two such intervals in $\R^2$.  
The $L^2(\R)$-based Sobolev space of order $s \in \R$  will be denoted by $H^s = H^s(\R) = (1-\Delta)^{-s/2}L^2$ 
as usual.  
 If $f:\R\times [0, T] \to \R$, 
 the mixed
 $L_T^q L_x^p$-norm of $f$ is 
\begin{equation*}
\|f\|_{L_T^q L_x^p} = \left(\int_0^T \left(\int_{\R} |f(x, t)|^p\,dx
\right)^{q/p}\,dt\right)^{1/q},
\end{equation*}
with the usual modification when $p$ or $q$ is $\infty$. An analogous definition is used for the other mixed norms $L_x^pL_T^q$,
with the order of integration in time and space interchanged.
In the notation $L_x^pL_T^q$  or $L_T^pL_x^q$, $T$ is replaced by $t$
 when the interval $[0, T]$ is instead the  whole real line $\R$.
For $T>0$ and $s\in \R$, $ C([0,T];H^s)$  denotes the space  of continuous maps from 
$[0,T]$ to $H^s$ with its usual norm,
$\|u\|_{C([0,T];H^s)}:= \sup_{t\in [0, T]}\|u(x, \cdot)\|_{H^s}$.

We use $c$ or $C$ to denote various  space- and time-independent constants whose exact values  may
 vary from one line to the next. 
The notation $ A \lesssim B$ connotes an estimate of the form $A\leq  cB$  for some $c$, while $A \sim B$ means $ A \lesssim B$  and $B \lesssim A$.  The notation $a+$ 
stands for 
 $a +\epsilon$  for  any $\epsilon > 0$, no matter how small.

Here are the main results. The first one is about the local well-posedness in $H^s (\R$), $s \geq 1$.

\begin{theorem}\label{mainTh1}  Assume $\gamma_1, \delta_1 >0$.  
For any $s\geq 1$ and for given  $\eta_0\in H^s(\R)$, there exist a time $T =T(\|\eta_0\|_{H^s})$
 and a unique function  $\eta \in C([0,T];H^s)$ which 
 is a solution of the IVP for (\ref{eq1.6}), posed with initial data $\eta_0$.  The solution $\eta$ 
varies continuously in $C([0,T];H^s)$ as $\eta_0$ varies in $H^s$.
\end{theorem}

With more regularity and a further restriction on the coefficients of the equation, global well-posedness holds, as the next theorem attests.

\begin{theorem}\label{mainTh3}  Assume $\gamma_1, \delta_1 > 0$.  
 Let $s\geq \frac32$ and $\gamma=\frac7{48}$.  Then  the solution to the IVP associated to (\ref{eq1.6})  given by Theorem \ref{mainTh1} can be extended to arbitrarily large time intervals $[0, T]$.  Hence the problem is globally well-posed in this case.
\end{theorem}

\section{Well-posedness Theory in $H^s$, $s\geq 1$}\label{sec4}

 Local well-posedness will be established using multilinear estimates combined with a contraction mapping argument.  Global well-posedness
in the spaces $H^s$ with $s \geq 2$ is obtained via energy-type arguments together with the local theory.  For values of $s$ below $2$, the 
global theory results from splitting the initial data into a small, rough part and a smooth part and writing evolution equations for each 
of these in such a way that the sum of the results of the separate evolutions provides a solution of the original problem.

\subsection{Local well-posedness}

This section will focus upon local well-posedness issues for the Cauchy problem associated to (\ref{eq1.6})
for given data  $\eta(x,0) =\eta_0(x)$ in $H^s(\R)$.  The first step is to write (\ref{eq1.6}) in an equivalent 
integral equation format. 
Taking the Fourier transform of equation (\ref{eq1.6}) with respect to the spatial variable yields
\begin{equation}\label{eq1.7}
\widehat\eta_t +i\xi\widehat\eta +\gamma_1 \xi^2\widehat\eta_t  - i \gamma_2 \xi^3\widehat\eta+\delta_1\xi^4\widehat\eta_t +\delta_2i\xi^5\widehat\eta+ \frac34 i\xi\widehat{\eta^2} -\gamma i\xi^3\widehat{\eta^2} -\frac18 i\xi \widehat{\eta^3} -\frac7{48} i\xi \widehat{\eta_x^2}=0, 
\end{equation}
or what is the same,
\begin{equation}\label{eq1.8}
\Big(1+\gamma_1\xi^2+\delta_1 \xi^4\Big)i\widehat\eta_t =\xi(1-\gamma_2\xi^2+\delta_2\xi^4)\widehat\eta+\frac14(3\xi-4\gamma\xi^3)\widehat{\eta^2}  -\frac18 \xi \widehat{\eta^3} -\frac7{48}\xi \widehat{\eta_x^2}.
\end{equation}

Because $\gamma_1,  \delta_1$ are taken to be positive, the fourth-order polynomial 
\begin{equation}\label{def-phi}
 \varphi(\xi) := 1 + \gamma_1\xi^2+\delta_1\xi^4,
\end{equation}
is strictly positive.  Define the three Fourier multiplier operators $\phi(\partial_x)$, $\psi(\partial_x)$ and $\tau(\partial_x)$ via
their symbols, {\it viz.}
\begin{equation}\label{phi-D}
\widehat{\phi(\partial_x)f}(\xi):=\phi(\xi)\widehat{f}(\xi), \qquad \widehat{\psi(\partial_x)f}(\xi):=\psi(\xi)\widehat{f}(\xi) \;\;\; {\rm and }\;\;\; \widehat{\tau(\partial_x)f}(\xi):=\tau(\xi)\widehat{f}(\xi), 
\end{equation}
 where
\begin{equation}
  \phi(\xi)=\frac{\xi(1-\gamma_2\xi^2+\delta_2\xi^4)}{\varphi(\xi)}, \quad \psi(\xi)=\frac{\xi}{\varphi(\xi)} \quad  {\rm and} \quad \tau(\xi)=\frac{3\xi-4\gamma\xi^3}{4\varphi(\xi)}.
\end{equation}

With this notation, the Cauchy problem associated to  equation (\ref{eq1.6}) can be written in the form
\begin{equation}\label{eq1.9}
\begin{cases}
i\eta_t = \phi(\partial_x)\eta + \tau (\partial_x)\eta^2 - \frac18\psi(\partial_x)\eta^3  -\frac7{48}\psi(\partial_x)\eta_x^2\, ,\\
 \eta(x,0) = \eta_0(x).
 \end{cases}
\end{equation}
Consider first the linear IVP 
\begin{equation}\label{eq1.10}
\begin{cases}
i\eta_t = \phi(\partial_x)\eta,\\
\eta(x,0) = \eta_0(x),
\end{cases}
\end{equation}
whose solution is given  by $\eta(t) = S(t)\eta_0$, where $\widehat{S(t)\eta_0} = e^{-i\phi(\xi)t}\widehat{\eta_0}$ is defined via its Fourier transform.
Clearly, $S(t)$ is a unitary operator on $H^s$ for any $s \in \R$, so that
\begin{equation}\label{eq1.11}
\|S(t)\eta_0\|_{H^s} = \|\eta_0\|_{H^s},
\end{equation}
for all $t > 0$.
Duhamel's formula allows us to write the IVP  (\ref{eq1.9}) in the equivalent integral equation form,
\begin{equation}\label{eq1.12}
\eta(x,t) = S(t)\eta_0 -i\int_0^tS(t-t')\Big(\tau(\partial_x)\eta^2 - \frac18 \psi(\partial_x)\eta^3 -\frac7{48}\psi (\partial_x)\eta_x^2\Big)(x, t') dt'.
\end{equation}

In what follows, a short-time solution of (\ref{eq1.12}) will be obtained via the contraction mapping principle in the space $C([0,T];H^s)$.  This will provide a 
proof of Theorem \ref{mainTh1}.

\subsubsection{Multilinear Estimates}

 Various multilinear estimates are now established that will be useful in the proof of the local well-posedness result.
  First, we record the  following ``sharp" bilinear estimate obtained in \cite{BT}.

\begin{lemma}\label{BT1}
 For $s \ge 0$, there is a constant $C = C_s$ for which
\begin{equation}\label{bt}
\|\omega(\partial_x) (u v)\|_{H^s} \le C\|u\|_{H^s}\|v\|_{H^s}
\end{equation}
 where $\omega(\partial_x)$  is the Fourier multiplier operator 
with symbol
\begin{equation} \label{btx0}
\omega(\xi) \, = \, \frac{|\xi|}{1 + \xi^2}.
\end{equation} 
\end{lemma}

It is worth noting that there is a counterexample in \cite{BT} showing that the inequality (\ref{bt}) is false if $s<0$.

\begin{corollary}\label{Lema1}
For any $s \ge 0$, there is a constant $C = C_s$ such that the inequality
\begin{equation}\label{bilin-1}
\|\tau(\partial_x) \eta^2\|_{H^s} \le C \| \eta\|_{H^s} ^2
\end{equation}
holds, where the operator $\tau(\partial_x)$  is  defined in (\ref{phi-D}).
\end{corollary}
\begin{proof}
Since  $\delta_1>0$, it follows  that $\tau(\xi) \leq C \omega(\xi)$ for some constant $C>0$.  The proof of
the estimate (\ref{bilin-1}) thus follows from Lemma \ref{BT1}.
\end{proof}

\begin{proposition}\label{P1}
For $s \ge \frac16$, there is a constant $C = C_s$ such that 
\begin{equation}\label{trilin-1}
\|\psi(\partial_x) \eta^3\|_{H^s} \le C \| \eta\|_{H^s} ^3.
\end{equation}
\end{proposition}
\begin{proof}
Consider first when $\frac16 \le s < \frac52$.   In this case, it appears that 
\begin{equation}\label{x2}
\Big|(1+|\xi|)^s \,\psi(\xi)\Big|=\Big|\frac{ (1+|\xi|)^s \xi}{(1 +\gamma_1 \xi^2+\delta_1\xi^4)}\Big|
 \le C \frac{1}{(1+|\xi|)^{3-s}}.
\end{equation}
 The last inequality implies that
\begin{equation}\label{x11}
\begin{split}
\|\psi(\partial_x) \eta^3\|_{H^s} &= \|(1+|\xi|)^s \,\psi(\xi)\widehat{\eta^3}(\xi)\|_{L^2}
\le C\left\|\frac{1}{(1+|\xi|)^{3-s}} \widehat{\eta^3}(\xi)\right\|_{L^2}\\ &\leq C \left\|\frac{1}{(1+|\xi|)^{3-s}}\right\|_{L^2}\| \widehat{\eta^3}\|_{L^{\infty}} \leq C\|\eta\|_{L^3}^3.
\end{split}
\end{equation}
In one dimension, the Sobolev embedding theorem states in part that $H^{\frac16}$ is embedded in $L^3$, so 
\begin{equation}\label{x3}
\|\eta\|_{L^{3}} \le C \|\eta\|_{H^\frac16},
\end{equation}
whence 
$$
\|\psi(\partial_x) \eta^3\|_{H^s} \le  C\| \eta\|_{H^s} ^3
$$
whenever $\frac16 \leq s < \frac52$.  

On the other hand, if $s >1/2$,  the Sobolev space $H^s$ is a Banach algebra.  
Since $|\psi(\xi)| \leq C\omega(\xi)$, Lemma \ref{BT1} implies that
\begin{equation}
\|\psi(\partial_x) (\eta \eta^2)\|_{H^s} \le C\|\eta \|_{H^s} \|\eta^2
\|_{H^s}\le C\|\eta \|_{H^s}^3,
\end{equation}
which completes the proof of  Proposition \ref{P1}.
\end{proof}

\begin{remark}   The reader will appreciate presently that this result is 
only used in case $s > \frac12$, so the full power of the last proposition 
is not needed in our theory.  We thought it interesting that the result 
holds down to $s = \frac16$ and note that the inequality useful at this
level  could be in the setting of internal waves in the deep ocean.   
This  point will be investigated in future research.   

\end{remark}

\begin{lemma}
For $s \ge 1$, the inequality 
\begin{equation}\label{Sharp1}
\|\psi(\partial_x) \eta_x^2\|_{H^s} \le C \| \eta\|_{H^s} ^2
\end{equation}
holds.
\end{lemma}
\begin{proof}
Observe that
$$
\psi(\xi) \leq C\omega(\xi) \frac1{1+ |\xi|}.
$$
The inequality (\ref{bt})  then allows the conclusion 
\begin{equation}\label{x6}
\|\psi(\partial_x) \eta_x^2\|_{H^s}  \leq C\| \omega(\partial_x) \eta_x^2\|_{H^{s-1}} \le  C\|  \eta_x\|_{H^{s-1}}\|  \eta_x\|_{H^{s-1}}
\le  C\|  \eta\|_{H^{s}}^2,
\end{equation}
since  $s-1 \ge 0$.
\end{proof}

The preceding ingredients are assembled to provide a proof of the local well-posedness theorem.
\begin{proof}[Proof of Theorem \ref{mainTh1}]
 Define a mapping
\begin{equation}\label{eq3.42}
\Psi\eta(x,t) = S(t)\eta_0 -i\int_0^tS(t-t')\Big(\tau(D_x)\eta^2 - \frac14 \psi(\partial_x)\eta^3
 -\frac7{48} \psi(\partial_x)\eta_x^2\Big)(x, t') dt'.
\end{equation}
The immediate goal is to show that this mapping is a contraction on a closed ball $\mathcal{B}_r$ with radius $r > 0$ and center at the origin in $C([0,T];H^s)$.

As remarked earlier,  $S(t)$ is a unitary group in $H^s(\R)$  (see (\ref{eq1.11})), and therefore
\begin{equation}\label{eq3.43}
\|\Psi\eta\|_{H^s} \leq \|\eta_0\|_{H^s} +CT\Big[\big{\|}\tau(\partial_x)\eta^2 - \frac18 \psi(\partial_x)\eta^3
 -\frac7{48}\psi(\partial_x)\eta_x^2\big{\|}_{C([0,T];H^s)}\Big].
\end{equation}
The inequalities (\ref{bilin-1}), (\ref{trilin-1}) and (\ref{Sharp1}) lead immediately to 
\begin{equation}\label{eq3.44}
\|\Psi\eta\|_{H^s} \leq \|\eta_0\|_{H^s} +CT\Big[\big{\|}\eta\big{\|}_{C([0,T];H^s)}^2 + \big{\|}\eta\big{\|}_{C([0,T];H^s)}^3 +\big{\|}\eta\big{\|}_{C([0,T];H^s)}^2\Big].
\end{equation}
If, in fact,  $\eta\in \mathcal{B}_r$, then (\ref{eq3.44}) yields
\begin{equation}\label{eq3.45}
\|\Psi\eta\|_{H^s} \leq \|\eta_0\|_{H^s} +CT\big[2r +r^2 \big]r.
\end{equation}

If we choose $r= 2\|\eta_0\|_{H^s}$ and $T= \frac1{2Cr(2 + r) }$, then  $\|\Psi\eta\|_{H^s} \leq r$, showing that $\Psi$ maps
 the closed ball $\mathcal{B}_r$ in $C([0,T];H^s)$ onto itself.   With the same choice of $r$ and $T$ and the same sort of 
estimates, one discovers that $\Psi$ is a contraction on $\mathcal{B}_r$ with contraction constant equal to $\frac12$ as it happens. The rest of
the proof is standard.
\end{proof}

\begin{remark}\label{rm2.1}
The following points follow immediately from the proof of the Theorem \ref{mainTh1}:
\begin{enumerate}
\item The maximal existence time $T=T_s$ of the solution satisfies
\begin{equation}\label{r2.45}
T\geq \bar{T} = \frac1{8C_s\|\eta_0\|_{H^s}(1+\|\eta_0\|_{H^s})},
\end{equation}
where the constant $C_s$ depends only on $s$.
\item The solution cannot grow too much, which is to say,
\begin{equation}\label{r2.46}
\|\eta(\cdot,t)\|_{H^s} \leq r =  2\|\eta_0\|_{H^s}
\end{equation}
for all $t \in [0,\bar{T}]$ where $\bar{T}$ is as above in (\ref{r2.45}).
\end{enumerate}
\end{remark}

\subsection{Global well-posedness}

In this section, {\em a priori} deduced bounds are obtained with an eye toward extending the local well-posedness just established.  The present theory 
  countenances the spaces $H^s(\R)$, $s \geq\frac32$.   However, we begin with  a global well-posedness result in $H^s(\R)$ for $s\geq 2$.

\subsubsection{Global well-posedness in $H^2$}
The aim here is to derive an {\em a priori} estimate in $H^2(\R)$, subject to certain restrictions on the parameters that appear in (\ref{eq1.6}). Multiplying the equation (\ref{eq1.6}) by $\eta$, integrating over the spatial domain $\R$ and integrating by parts yields
\begin{equation}
\frac{1}{2} \frac{d}{dt}  \int_{\mathbb{R}} \left(\eta^2 + \gamma_1\eta_x^2+\delta_1\eta_{xx}^2\, \right)dx+ \gamma \int_{\mathbb{R}}(\eta^2)_{xxx}\, \eta\,dx-\frac7{48} \int_{\mathbb{R}} (\eta_x^2)_x\, \eta\,dx=0.
\end{equation}
Further integrations by parts gives
\begin{equation}\label{apriori1}
\frac12 \frac{d}{dt} \int_{\mathbb{R}} \left(\eta^2 + \gamma_1\eta_x^2+\delta_1\eta_{xx}^2\, \right)dx= \left(\gamma -\frac7{48}\right)\int_{\mathbb{R}}\eta_x^3\,dx.
\end{equation}
Of course, these calculations involve derivatives of higher order than are guaranteed to exist by assuming the initial data lies only in $H^2$.  Moreover the term on 
the right-hand side of (\ref{apriori1}) does not converge if the function $\eta$ only lies in $H^2$.  However,  one can make the calculations using smoother solutions and then pass to the limit of rougher data making use of the continuous dependence result.  The idea is standard 
and we pass over the details (cf. \cite{BK}).  

From (\ref{apriori1}) it is clear that an {\em a priori} estimate obtains when $\gamma =\frac7{48}$.  That such a condition 
can be imposed while respecting the other mathematical limitations $\gamma_1>0$ and  $\delta_1 > 0$ will be discussed in 
Section \ref{sec5}.   For the time being,
we presume that $\theta, \lambda, \mu, \lambda_1$, $\mu_1$ 
and $\rho$  have been chosen so that $\gamma = \frac7{48}$
 and $\gamma_1, \delta_1 > 0$ still holds.  
In this case, the equation (\ref{eq1.6}) becomes
\begin{equation}\label{eta}
\eta_t+\eta_x-\gamma_1 \eta_{xxt}+\gamma_2\eta_{xxx}+\delta_1\eta_{xxxxt}+\delta_2\eta_{xxxxx}+\frac34(\eta^2)_x+\gamma\big(\eta^2\big)_{xxx}-\gamma\big(\eta_x^2\big)_x -\frac18\big(\eta^3\big)_x=0.
\end{equation} 
In this form, it has the conserved quantity
\begin{equation}\label{apriori4}
 E(\eta(\cdot,t)):=  \frac12\int_{\mathbb{R}} \eta^2 + \gamma_1 (\eta_x)^2+\delta_1(\eta_{xx})^2\, dx= E(\eta_0).
\end{equation}

\begin{remark}  \label{hamiltonian}
In fact, with the restriction $\gamma =\frac7{48}$, the equation is Hamiltonian, for there is a second conserved quantity, namely 
\begin{equation}  \label{invariant2}
\Theta(\eta) = \frac12\int_\R \left( -\eta^2 -\frac12\eta^3 + \frac1{16}\eta^4 + \frac{7}{24} \eta \eta_x^2+\gamma_2\eta_x^2 -  \delta_2 \eta_{xx}^2\right) dx.
\end{equation}
   The system itself may be written in the 
Hamiltonian format
\begin{equation}\label{Hamiltonian}
  \frac{\partial}{\partial t} \nabla E(\eta) \, = \, \frac{\partial}{\partial x} \nabla \Theta(\eta)
\end{equation}  
where $\nabla E$ is the Euler derivative of $E$ and  similarly $\nabla \Theta$ the Euler derivative of $\Theta$.  
\end{remark}

The conservation law (\ref{apriori4}), which is essentially the $H^2$--norm, immediately points to the following global well-posedness result.

\begin{theorem}\label{mainTh2}
 Let $s\geq 2$ and suppose $\gamma_1, \delta_1 > 0$ and $\gamma=\frac7{48}$.   Then  the IVP for equation (\ref{eq1.6})  is globally well-posed in $H^s(\R)$.
\end{theorem}
\begin{proof}
Following a standard argument, the global well-posedness in $H^2(\R)$ is a consequence of the local theory and the {\it a priori} bound implied by the conserved quantity (\ref{apriori4}). To prove global well-posedness in $H^k$,  where $k \ge 3$ is an integer, proceed by induction on $k$.  

Assume that $\eta_0$ lies in $H^3$.  The local well-posedness theory then delivers a solution in $C([0,T];H^3)$ for some $T > 0$.  If {\it a priori} bounds on the $H^3$--norm of $\eta$ which are finite on finite time intervals holds, 
then the local theory can be iterated and a global solution results.  

Differentiate  equation (\ref{eta}) with respect to the spatial variable, multiply the resulting equation by $\eta_x$ and integrate over $\R$.  After integrations by parts in the spatial variable, there obtains
\begin{equation}\label{eta-2}
\begin{split}
\frac12 \frac{d}{dt}  \int_{\mathbb{R}} \left(\eta_x^2 + \gamma_1 \eta_{xx}^2+\delta_1\eta_{xxx}^2 \right)\,dx +\frac34 \int_{\mathbb{R}}\eta_{x}^3 \, dx \\ -3 \gamma \int_{\mathbb{R}}\eta_{xx}^2 \eta_{x}\,dx-\frac38 \int_{\mathbb{R}} \,  \eta_x^3\,\eta\, dx=0.
\end{split}
\end{equation}
Standard Sobolev embedding results show that  for any time $t$ for which the solution exists,
\begin{equation}\label{immers11}
\begin{split}
\|\eta\|_{L_x^2}^2 \le 2 E_0, \quad \|\eta_x\|_{L_x^2}^2 \le \frac2{\gamma_1}E_0, \quad \|\eta_{xx}\|_{L_x^2}^2 \le \frac{2}{\delta_1}E_0, \\
 \|\eta\|_{L_x^{\infty}}^2 \le \frac4{\sqrt{\gamma_1}}\,E_0, \quad \|\eta_x\|_{L_x^{\infty}}^2 \le \frac{4}{\sqrt{\delta_1\gamma_1}}\,\,E_0,
\end{split}
\end{equation}
where $E_0 = E(\eta_0)$.
After integrating (\ref{eta-2})  with respect to time over the interval $[0,t]$, making elementary 
estimates of all the terms not involving a third derivative  and using (\ref{immers11}) systematically, there obtains the inequality 
\begin{eqnarray}
\delta_1\int_{\mathbb{R}}  \eta_{xxx}^2\, dx& \leq & \,\int_{\mathbb{R}} \left( (\eta_{0x})^2+\gamma_1(\eta_{0xx})^2 + \delta_1(\eta_{0xxx})^2\right)\,dx \nonumber\\ 
 &\,& +\,C \int_0^t \|\eta_x\|_{L_x^{\infty}} \left( \|\eta_x\|_{L_x^2}^2+ \|\eta_{xx}\|_{L_x^2}^2+\|\eta_x\|_{L_x^2}^2\,\|\eta\|_{L_x^{\infty}}\right) dx\nonumber \\
& \le &\,\delta_1\int_{\mathbb{R}} (\eta_{0xxx})^2\, dx+ CE_0 + C E_0^{3/2}\left( 1+ E_0^{1/2}\right) t,\nonumber
\end{eqnarray}
from which the desired $H^3$--bound follows.

Assuming there are in hand $H^k$ bounds, an entirely similar energy-type calculation reveals that the solution $\eta$ has  $H^{k+1}$--bounds 
as soon as the initial data $\eta_0$ lies in $H^{k+1}$.   We pass over the details.

To obtain global well-posedness in the fractional-order Sobolev spaces $H^s$, $s \ge 2$ not an integer, a straightforward application of nonlinear interpolation theory (see \cite{BSc}, \cite{BCW}) may be applied, thereby completing the proof of the theorem.
\end{proof}

\subsubsection{Global well-posedness in $H^s$, $s\geq \frac32$}

The object of this subsection is to prove the second main result, Theorem \ref{mainTh3}.  To establish well-posedness below the level where 
{\it a priori} bounds obtain, a Fourier splitting technique will be employed wherein the data $\eta_0$ is decomposed into a small, rough 
part and a smooth part.  As already mentioned, such decompositions are commonplace in various 
contexts in the theory of partial differential equations. 

Let there be given initial data $\eta_0 \in H^s$ where $1 \leq s < 2$ and a $T>0$.    
As advertised, the data $\eta_0$ is decomposed into a small part and a smooth part, {\it viz.}
\begin{equation} \label{decompose}
\eta_0 = w_0 \, + \, v_0 \qquad {\rm where} \qquad w_0 \in H^\infty \,\,\, {\rm and} \,\,\, v_0 \in {H^s}
\end{equation}
is small.  Such a decomposition can be effected in many ways.  One that is especially helpful 
in what follows is the one-parameter family $\{w_0^\epsilon\}_{\epsilon > 0}$ defined by way of their 
Fourier transforms to be 
\begin{equation} \label{regularize}
\widehat{w_0^\epsilon} = \zeta(\epsilon \xi)\widehat{\eta_0}(\xi)
\end{equation}
where $\zeta$ is an even, $C^\infty$--function defined on $\R, 0 \, \leq \zeta \leq 1, \, \zeta(0) = 1$ 
and such that $1-\zeta(\xi)$ has a zero of infinite order at $\xi = 0$ while $\zeta$ decays exponentially 
to 0 as $|\xi| \to \infty$. 
(For example, $\zeta$ could be a cut-off function which is identically equal  to 1 on the interval $[-1,1]$ 
and has support in $[-2,2]$.)
 It  follows by a straightforward computation in the Fourier transformed variables
 that if $\eta_0 \in H^s$, then for $r \geq 0$, 
\begin{eqnarray} \label{approximation}
\|w_0^\epsilon\|_{H^{s+r}} \,=\,  O\left(\epsilon^{-r} \right) \quad {\rm and}\qquad
\|\eta_0 - w_0^\epsilon\|_{H^{s-r}}  \, =  \,o\left(\epsilon^r\right)
\end{eqnarray}
as $\epsilon \downarrow 0$ (see, for example, Lemma 5 in \cite{BS}).  
Define $v_0 = v^\epsilon_0 =  \eta_0 - w_0^{\epsilon}$.  For the moment, the 
 dependence of both $v_0$ and $w_0$ upon $\epsilon$ will be suppressed.  
The values of $\epsilon$ will be appropriately limited  presently.  

By choosing $\epsilon$ small enough so that  $\|v_0\|_{H^s} \leq 1$ and $\|v_0\|_{H^s} \leq \frac1{12C_sT}$,   
 the local well-posedness theory adduced in Theorem \ref{mainTh1}
assures us that if we pose $v_0$ as initial data 
for our evolution equation (\ref{eta}), then the solution $v$ emanating from it will  lie in $C([0,T];H^s)$ and it will 
not be larger than  $2\|v_0\|_{H^s}$ over the entire time interval $[0,T]$  (see Remark \ref{rm2.1}).  We can also insure that  
\begin{equation} \label{H1small}
\|v(\cdot,t)\|_{H^1} \leq 2\|v_0\|_{H^1}\quad  {\rm for \, \, all} \,\,  t \in [0,T], 
\end{equation}
simply by imposing the further restriction $\|v_0\|_{H^1} \leq \frac1{12C_1T}$.  This follows since the integral operator 
$\Psi$ in (\ref{eq3.42}) will simultaneously satisfy (\ref{r2.45}) and (\ref{r2.46}) for both 
the Sobolev indices $s$ and $1$.  
The solutions, which are the fixed points of $\Psi$ in the two spaces, must be the same by uniqueness in the larger space.

Once $v$ is fixed and known to exist on the entire time interval $[0,T]$, the smooth part $w_0$ of the initial data is evolved according to the variable coefficient
IVP
\begin{equation}\label{w-1}
\begin{cases}
w_t+w_x -\gamma_1 w_{xxt}+\gamma_2 w_{xxx}+\delta_1 w_{xxxxt}+\delta_2 w_{xxxxx}+G(v,w)=0, \\
 w(x,0)=w_0(x),
\end{cases}
\end{equation}
where  
 \begin{eqnarray}\label{3.4-1}
G(v,w)&\!\!\!\!\!:=\!\!\!\!& \frac32(v w)_x +\frac34(w^2)_x +2 \gamma(vw)_{xxx}+\gamma(w^2)_{xxx}\\
&\,& -2\gamma(v_xw_x)_{x}-\gamma(w_x^2)_{x}-\frac38(v^2w)_{x}-\frac38(vw^2)_{x}-\frac18(w^3)_{x}.\nonumber
\end{eqnarray}

If a solution $w$ exists in $C([0,T];H^s)$, then $v + w$ provides a solution on the time interval $[0,T]$ of the original 
problem for the equation (\ref{eta}) with initial value $\eta_0$.  As $T$ was arbitrary, global existence is thereby concluded.
Well-posedness then follows from the local theory.  That is, the continuous dependence of the
solution on the initial data and the uniqueness of solutions within the function class $C([0,T];H^s)$ 
derive from the previously elucidated local well-posedness results. Thus, Theorem \ref{mainTh3} will be established as soon as 
(\ref{w-1}) is shown to have a solution in $C([0,T];H^s)$.

\begin{proof}[Proof of Theorem \ref{mainTh3}]
As already discussed, the variable coefficient $v$ appearing in the nonlinearity (\ref{3.4-1}) lies in $C([0,T];H^s) \subset  C([0,T];H^1)$.  As a first step, it is important to show that the IVP (\ref{w-1}) for $w$ is locally well-posed in $H^2$ and not just in $H^s$.  
To this end, write the IVP (\ref{w-1}) in the equivalent, integral equation form   
\begin{equation}\label{w-integral}
\begin{cases}
w(x,t) = S(t)w_0 -i\int_0^tS(t-t')\Big(\tau(\partial_x)w^2 + 2\tau(\partial_x)wv - \frac18 \psi(\partial_x)w^3 \\ 
\hspace{1.75cm}
- \frac38 \psi(\partial_x)w^2v - \frac38 \psi(\partial_x)wv^2 
 - \gamma\psi(\partial_x)w_x^2  - 2\gamma\psi(\partial_x)w_xv_x  \Big)(x, t') dt' \\
\hspace{1.35cm} = \,\, \Phi(w)(x,t),
\end{cases}
\end{equation}
where the Fourier multiplier operators $\psi(\partial_x)$ and $\tau(\partial_x)$ are  as defined already in (\ref{phi-D}) and the unitary 
family $S(t)$ is the solution  group for the linear equation (\ref{eq1.10}).  
 
This integral equation is studied in $C([0,T];H^2)$ when the variable coefficient $v$ lies in $C([0,T];H^s)$.  As $w_0$ lies in $H^\infty$, it is clear that $S(t)w_0$ lies in $C(\R;H^2)$.   Just as in the earlier analysis of the integral equation (\ref{eq1.12}), the argument proceeds by showing that the 
mapping $ w \mapsto \Phi(w)$ defined by the right-hand side  of (\ref{w-integral}) is a contraction on a ball $\mathcal{B}_r$ of radius $r$ about $0$ in the space 
$C([0,T_0];H^2)$ for $r$ large enough and $T_0$ small enough.  This will establish the local well-posedness needed for the next step in
the analysis.  

The summands in the integral equation that only feature $w$ may be handled just as before and suitable estimates are forthcoming since 
we are working in $H^2$ (see the proof of Theorem \ref{mainTh1}).  The following lemma provides the extra information needed to complete the argument in favor of $\Phi$ being 
a contraction mapping on $\mathcal{B}_r \subset C([0,T_0];H^2)$ for suitable $T_0$ and $r$.  

\begin{lemma}  Suppose $1 \leq s < 2$.  Then for $f \in H^s$ and $g \in H^2$, there are constants $C$ depending only on $s$ such that
\begin{equation} \label{smoothing}
\begin{split}
\|\tau(\partial_x)fg\|_{H^2} \leq C \| f \|_{H^s} \|g  \|_{H^2}, \qquad  \|\psi(\partial_x)f^2g  \|_{H^2} \leq C \|f  \|^2_{H^s} \| g \|_{H^2}, \\
 \|\psi(\partial_x)fg^2  \|_{H^2} \leq C \|f  \|_{H^s} \| g \|^2_{H^2}, \qquad \|\psi(\partial_x)f_xg_x  \|_{H^2} \leq C \|f  \|_{H^s} \| g \|_{H^2}.
 \end{split}
\end{equation}
\end{lemma}

\begin{proof}  As $\tau(\partial_x)$ is a bounded map from $H^r$ to $H^{r+1}$, it follows that
$$
\|\tau(\partial_x)fg\|_{H^2} \leq C\|fg \|_{H^1}   \leq  C\|f \|_{H^1} \|g \|_{H^1} \leq C\|f \|_{H^s} \|g \|_{H^2.} 
$$
The operator $\psi(\partial_x)$ maps $H^r$ to $H^{r+3}$.  Consequently, we have
\begin{eqnarray}
\|\psi(\partial_x) f^2g \|_{H^2} \leq C\|f ^2g\|_{H^1}  \leq C\|f \|^2_{H^1} \|g \|_{H^1}  \leq C\|f \|^2_{H^s}\|g \|_{H^2}, \nonumber\\
\|\psi(\partial_x) fg^2 \|_{H^2} \leq C\|f g^2\|_{H^1}  \leq C\|f \|_{H^1} \|g \|^2_{H^1}  \leq C\|f \|_{H^s}\|g \|^2_{H^2},\nonumber \\
 \|\psi(\partial_x)f_xg_x \|_{H^2} \leq C\|f_x g_x\|_{L^2}  \leq C\|f_x \|_{L^2} \|g_x \|_{L^\infty}  \leq C\|f \|_{H^s}\|g \|_{H^2},\nonumber
\end{eqnarray}
and the results are established.  
\end{proof}

It is straightforward to use the smoothing estimates (\ref{smoothing}) to show that the mapping $\Phi$ 
 is a contraction on a suitably chosen ball about the origin in $C([0,T_0];H^2)$ 
for $T_0$ small enough, which is the content of the following proposition.  

\begin{proposition}
The IVP (\ref{w-1}) is locally well-posed in $H^2$.  
\end{proposition}

It remains only to show that the local in time solution $w$ of (\ref{w-1}) can be continued to the entire time interval $[0,T]$.  
This in turn will be settled as soon as  {\it a priori} bounds on $w$ in $H^2$ 
are provided which are  valid on $[0,T]$.     
To see such a bound obtains, 
multiply  equation (\ref{w-1}) by $w$, integrate over $\R$ and  integrate by parts in the spatial variable to obtain   
\begin{equation}\label{eqomega}
\begin{split}
\frac12 \frac{\partial}{\partial t} & \int_{\mathbb{R}} \left(w^2  +\gamma_1 w_x^2+\delta_1 w_{xx}^2\,\right)dx \,=\,
\frac32\int_{\mathbb{R}} v w w_x dx-2\gamma \int_{\mathbb{R}}(v w)_{x}w_{xx}\,dx\\
& \qquad -2\gamma \int_{\mathbb{R}}\,v_x (w_x)^2\, dx-\frac38 \int_{\mathbb{R}}v^2 w\,w_{x}\, dx-\frac38 \int_{\mathbb{R}}\,v\, w^2 \, w_x\,dx.
\end{split}
\end{equation}
The intermediate computations are justified as before by use of the continuous dependence results in $H^2$ for $w$ and $H^1$ for $v$.  Let $\mathcal{X}(t):=\int_{\mathbb{R}} \left(w^2 +\gamma_1 w_x^2+\delta_1 w_{xx}^2\,\right)dx$.  Then, $\mathcal{X}(t)$ is equivalent to the
square of the 
$H^2$--norm of $w(\cdot,t)$.   

The next task is to obtain an upper bound on the right-hand side of (\ref{eqomega}) in terms of $\|w\|_{H^2}$ and $\|v\|_{H^1}$.  
The fact that $\|w\|_{L^\infty}$ and   $\|w_x\|_{L^\infty}$ are both bounded by $\|w\|_{H^2}$ and elementary estimates imply that
\begin{equation} \label{inequality1}
\begin{split}
\frac{\partial  \mathcal{X}(t)}{\partial t}  & \leq C\Big(\big(\|v\|_{H^1} + \|v\|^2_{H^1}\big)\|w\|^2_{H^2}
 + \|v\|_{H^1}\|w\|^3_{H^2}\Big) \\ 
 & \leq C\Big(\big(\|v\|_{H^1} + \|v\|^2_{H^1}\big)\mathcal{X}(t) 
+ \|v\|_{H^1}\mathcal{X}(t)^{\frac32}\Big).
\end{split}
\end{equation}
Recall that $\|v(\cdot,t)\|_{H^1}  \leq 2\|v_0\|_{H^1}$ on the entire interval $[0,T]$.  In consequence, (\ref{inequality1}) 
can be extended thusly;
\begin{equation}  \label{inequality2}
\frac{\partial  \mathcal{X}(t)}{\partial t} \, \leq \, 2C\|v_0\|_{H^1}\left(\mathcal{X}(t) + \mathcal{X}(t)^{\frac32} \right). 
\end{equation}
Notice that, because of   (\ref{approximation}),  
\begin{equation}  \label{v0}
\|v_0\|_{H^1} = o(\epsilon^{s-1}) = \nu(\epsilon)\epsilon^{s-1} \quad  
   {\rm where} \quad \nu(\epsilon) \to 0 \quad  {\rm as} \quad \epsilon \to 0.
\end{equation}
If $\Sigma(t)$ is the solution of 
\begin{equation} \label{sigma}
\frac{d\Sigma}{dt} = 2C\|v_0\|_{H^1}\left(\Sigma(t) + \Sigma(t)^{\frac32}\right)
\end{equation}
with $\Sigma(0) = \mathcal{X}(0)$, then a Gronwall-type argument implies that $\mathcal{X}(t) \leq \Sigma(t)$ 
for all $t$ for which $\Sigma$ is finite.    The solution of (\ref{sigma}) is 
\begin{equation} \label{soln}
\sigma(t) = \frac{\sigma(0)e^{Ct\|v_0\|_{H^1}}}{1 - \sigma(0)\left(e^{Ct\|v_0\|_{H^1}} - 1\right)} \leq 
\frac{\sigma(0)e^{CT\|v_0\|_{H^1}}}{1 - \sigma(0)\left(e^{CT\|v_0\|_{H^1}} - 1\right)},
\end{equation}
provided the right-hand side is positive and finite, where $\sigma(t)^2 = \Sigma(t)$.  Of course, as long as $0 \leq y  \leq 1$, 
say, then $e^y - 1 \leq ey$.  Since $T$ is fixed and $\|v_0\|_{H^1}$ is small for small values of $\epsilon$, the right-hand side
of (\ref{soln}) may be bounded above by
\begin{equation}  \label{w-bound}
\frac{\sigma(0)e^{CT \|v_0\|_{H^1}}}{1 - CTe\sigma(0)\|v_0\|_{H^1}}.
\end{equation}
 The latter will provide the desired upper bound needed to continue the solution $w$ to the entire time interval $[0,T]$
as soon as 
\begin{equation} \label{small}
\sigma(0) \|v_0\|_{H^1} < \frac1{CeT}.
\end{equation}  
   As $\sigma(0)$ is equivalent to the $H^2$--norm of $w_0$,  (\ref{approximation}) implies that 
$
\sigma(0) \, \leq \, C\epsilon^{s-2}.
$
Combining this with (\ref{v0}), it is seen that 
$$
\sigma(0)\|v_0\|_{H^1} \, = \, o\left(\epsilon^{2s - 3} \right) \quad {\rm as} \quad \epsilon \downarrow 0.
$$ 
Consequently, if $s \geq \frac32$ and $\epsilon$ small enough, (\ref{small}) is valid and the result is proved.  
\end{proof}

\section{Parameter Restrictions}\label{sec5}

The class of partial differential equations (\ref{eq1.6}) are all formally equivalent models for long-crested, small amplitude, long waves 
on the surface of an ideal fluid over a flat bottom.   The hope is that they approximate solutions of the
full water-wave problem for an ideal fluid with an error that is of order $O\left(\beta^3 t\right)$ over a time scale at least of 
order $O\left(\beta^{-2}\right)$.   Rigorous theory to this effect, but only on the shorter, Boussinesq 
time scale $O\left(\beta^{-1}\right)$, is available for the lower order, unidirectional models (\ref{kdvbbm}) 
by combining results in   \cite{AABCW},
  \cite{BCL1} and \cite{BPS1}. 

It deserves remark that various models already existing in the literature are
specializations of the class of models displayed in \eqref{eq1.5}.  For example,
the model derived in \cite{dullin} in it's zero surface tension limit, and see also \cite{johnson} and 
\cite{marchant}, appears by taking $\rho = b + d$ and an appropriate 
choice of $\lambda_1$.  As will be clear momentarily, 
this model, like the one in \cite{olver}, is not Hamiltonian.   

Despite the fact that the models are formally equivalent, they may have very different mathematical
properties.  When it comes time to choose one of the models for use in a real-world situation, one naturally
wants to have good mathematical properties at hand.  This was discussed in some detail in \cite{BCS1} and 
\cite{BCS2} in the context of the lower-order system (\ref{b1.1})--(\ref{abcd}).  

In the present account, theory has been developed that implies the local well-posedness of the initial-value problem 
for a subclass of our unidirectional models.  Local well-posedness is a minimal requirement for the use of such
models in practice.  We also found an additional condition which allows the local theory to be continued 
indefinitely.  It is especially noteworthy that this condition implies the equation to have a Hamiltonian structure.
The full water wave model also has a Hamiltonian structure, and experience indicates that maintaining such a 
Hamiltonian arrangement in approximate models is likely to lead to better qualitative agreement 
with the full model.  Hence, our recommendation is to use the special versions of our equation 
displayed in (\ref{eta}).

Interest is now turned to specifying conditions under which the various restrictions on the coefficients $\gamma_1$, 
$\delta_1$ and $\gamma$ that cropped up during
our analysis are valid.  Recall that these conditions were 
\begin{equation} \label{param}
\gamma_1\,>\,0, \quad \delta_1 \, > \, 0\quad {\rm and} \quad \gamma = \frac7{48}
\end{equation}
(see Theorem \ref{mainTh2}).  The models satisfying these three conditions appear to have a more satisfactory mathematical theory.    It is worth reiterating that comparison results   
indicating 
that such models approximate solutions of the full water-wave problem rely upon 
smoothness (see \cite{AABCW}, \cite{BCL1}  and \cite{DL}, for example).   The fact that, with the
restrictions (\ref{param}), the model is globally well-posed in smooth function classes is therefore
 potentially very useful.  
 
 
\subsection{\it Hamiltonian Structure}  The Hamiltonian structure displayed in Remark \ref{hamiltonian} is the key to our global well-posedness results.   It also engenders other good features in the model which are not entered upon here.

So far, the condition $\gamma = \frac{7}{48}$ is the only one for which we know existence of a Hamiltonian structure.  Looking at the formula for $\gamma$ given in (\ref{gamas})  and demanding that $\gamma =\frac7{48}$  implies that
\begin{equation}\label{theta-6ml}
\frac1{24}\big[5-9(b+d)+9\rho\big] = \frac7{48}.
\end{equation}
Thus, the Hamiltonian structure is guaranteed if one chooses $\rho$ by the  formula
\begin{equation}\label{rho-2}
\rho = b+d-\frac16,
\end{equation}
which is exactly the one  advertised in \eqref{correction-1}.
In terms of the fundamental parameters $\theta, \lambda$ and $\mu$, $\rho$ given in \eqref{rho-2} is written as
\begin{equation}\label{rho-3}
\rho =  \frac16\big[1-\big(\theta^2-\frac13\big)\lambda-3\big(1-\theta^2)\mu\big] = \frac16-(a+c),
\end{equation}
  where the relation $a+b+c+d=\frac13$ has been used.

\subsection{\it Well-Posedness}    As mentioned already, equation (\ref{eq1.6}) is easily seen to be linearly ill-posed in Sobolev 
classes unless the parameters $\gamma_1$ and $\delta_1$ are positive.  These are the more important of the three restrictions in \eqref{param}  as far as 
well-posedness is concerned.  We fix the value of $\rho = b + d -\frac16$ given by \eqref{rho-2} for which
 $\gamma_1 = \gamma_2=\frac1{12}$. In particular, $\gamma_1 > 0$, so 
that condition is met.   In what follows, we discuss the condition $\delta_1>0$.

As noted in Remark \ref{rem-2.1}, a straightforward calculation reveals that 
\begin{equation}  \label{coefficient}
\delta_2 - \delta_1 + \frac16 \gamma_1 \, = \, \frac{19}{360},
\end{equation}
 regardless of the choice of the various fundamental parameters.   As $\gamma_1 = 1/12$, 
it is further deduced that 
\begin{equation}   \label{del2}
\delta_2 \,=\, \delta_1 +  \frac{7}{180}.
\end{equation}
Thus, the condition $\gamma = 7/48$ implies \eqref{rho-2}.   This in turn yields 
\eqref{del2}.   So, any value of $\delta_1 > 0$ may be specified as long as it is 
consistent with choices of $\theta, \lambda, \mu, \lambda_1$ and $\mu_1$.  

Using the formula (\ref{gamas}) for $\delta_1$ together with the formulas (\ref{abcd}) and (\ref{abcd1}) for the
coefficients $a, b, \cdots, c_1, d_1$ and (\ref{rho-2}) for $\rho$, a little algebra shows that in terms of the fundamental parameters 
$\theta, \lambda, \mu$, $\lambda_1$ and $\mu_1$, 
\begin{equation}\label{delta}
\begin{split}
\delta_1 =\delta_1(\theta, \lambda, \mu, \lambda_1,\mu_1)& =\frac12(b_1+d_1)-\frac14\Big(2b-\frac16\Big)\Big(\frac16-a-d\Big)-\frac14d\Big(\frac16-2a\Big)\\
&= -\frac5{48}\big(\theta^2-\frac15\big)\Big[\big(\theta^2-\frac15\big)(1-\lambda_1)+(1-\theta^2)\mu_1\Big]\\
&\quad -\frac14\Big[\Big(\theta^2-\frac13\Big)(1-\lambda)-\frac16\Big]\Big[\frac16-\frac12\Big(\theta^2-\frac13\Big)\lambda -\frac12(1-\theta^2)(1-\mu)\Big]\\
 &\quad-\frac18(1-\theta^2)(1-\mu)\Big[\frac16-\Big(\theta^2-\frac13\Big)\lambda\Big]\\
&=\frac5{48}\big(\theta^2-\frac15\big)^2\lambda_1-\frac5{48}\big(\theta^2-\frac15\big)(1-\theta^2)\mu_1+P(\theta, \lambda, \mu),
\end{split}
\end{equation}
where 
\begin{equation}\label{Poly-1}
\begin{split}
P(\theta, \lambda, \mu)&=-\frac{(3\theta^2-1)^2}{72}\lambda^2 + \frac{(3\theta^2-1)(6\theta^2-1)}{144} \lambda -\frac{(1-\theta^2)}{24}\mu -\frac{(5\theta^4-30\theta^2+14)}{240},
 \end{split}
\end{equation} 
is a polynomial  in  $\theta$, $\lambda$ and $\mu$.  A study of  (\ref{delta}) reveals that  there 
are two separate cases to consider.   \vspace{.2cm}

\noindent
{\bf Case  1:}  $\theta \in [0,1] \setminus \{ \frac{1}{\sqrt{5}}\}$. 
In this case $\delta_1>0$ if and only if
\begin{equation}\label{inel1}
\lambda_1> \frac{(1-\theta^2)\mu_1}{\big(\theta^2-\frac15\big)}-\frac{48}{5}\frac{P(\theta, \lambda, \mu)}{\big(\theta^2-\frac15\big)^2}=:\mathcal{H}(\theta, \lambda, \mu, \mu_1).
\end{equation}
Since  $\mathcal{H}(\theta, \lambda, \mu, \mu_1)$ is finite for any given values of $\theta$, $\lambda$, $\mu$ and $\mu_1$, it is always possible to choose an appropriate $\lambda_1$  such that the inequality (\ref{inel1}) holds true.
Indeed, there are many choices that work.  
\vspace{.2cm}

\noindent
 {\bf Case  2:}   $\theta =  \frac{1}{\sqrt{5}}$. In this case 
\begin{equation}\label{parabola}
\begin{split}
\delta_1\Big(\frac{1}{\sqrt{5}}, \lambda, \mu, \lambda_1,\mu_1\Big)&=P\Big(\frac{1}{\sqrt{5}}, \lambda, \mu\Big)\\
&= -\frac{1}{450} \lambda^2 -\frac{1}{1800}\lambda -\frac{1}{30}\mu-\frac{41}{1200}.
\end{split}
\end{equation}
Observe that the quadratic equation
\begin{equation}
P\Big(\frac{1}{\sqrt{5}}, \lambda, \mu\Big)=0,
\end{equation}
defines a parabola facing downward. The region in $\lambda-\mu$ space where  $\delta_1=P\left(\frac{1}{\sqrt{5}}, \lambda, \mu\right)>0$ is the shaded region inside the 
parabola shown in the  Figure 1, {\it viz.}
\begin{figure}[h!]
  \centering
      \includegraphics[width=0.3\textwidth]{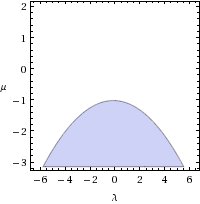}
  \caption{The region where  $P\left(\frac{1}{\sqrt{5}}, \lambda, \mu\right) > 0$ 
is shaded.  }
	\label{fig:A}
	\end{figure}



\section{The Dispersion Relation}  \label{sec-5}

  The models derived here 
depend upon choices of six parameters, which have 
been denoted $\lambda, \lambda_1, \mu, \mu_1, \theta$ and $\rho$. 
 The parameter $\theta$ has
physical significance whereas the others are modeling parameters and in principle, can take
any real value.  


 As will be seen in a moment, the linearized dispersion
relation for the class of models derived here always matches that of the full 
water-wave problem 
through second order in the small parameter $\beta$.  
More precisely, if any of these
 models are linearized about the rest state, the resulting linear 
partial differential equation has a dispersion relation relating phase speed $c$ to wave number $k$.  A brief calculation shows this to be 
\begin{equation}  
c_{model}(k) \, = \, 1 -  \big(\gamma_1 + \gamma_2\big)k^2 
+ \big(\delta_2 - \delta_1 + \gamma_1^2 
+ \gamma_1\gamma_2 \big) k^4+ \mathcal{F}k^6
\end{equation}
where $k$ is the wave number and the coefficient $\mathcal{F}$ is 
\begin{equation}   \label{F}
\mathcal{F} \, = \, \mathcal{F}(\theta,\lambda,\mu,\lambda_1,\mu_1, \rho) \, = \, 
-\gamma_1 \delta_2 -\gamma_2(-\delta_1 + \gamma_1^2) + 2 \gamma_1\delta_1 - \gamma_1^3.
\end{equation}
   As $\gamma_1 + \gamma_2 = 1/6$ holds independently of the 
choice of parameters, the second and third terms simplify, {\it viz.} 
\begin{equation}  \label{model dispersion}
c_{model}(k) \, = \, 1 -  \frac16 k^2 
+ \left(\delta_2 - \delta_1 + \frac16 \gamma_1 \right) k^4+ \mathcal{F}k^6,
\end{equation}
where the coefficient $\mathcal{F}$ will be displayed presently.
Making use of \eqref{del2} leads to the final result
\begin{equation}  \label{}
c_{model}(k) \, = \, 1 -  \frac16 k^2 
+ \left(\frac{19}{360}\right) k^4+ \mathcal{F}k^6,
\end{equation}
 regardless of the choice of the various parameters.

For the two-dimensional water wave problem displayed in (\ref{euler}), the linearized dispersion relationship is exactly
\begin{equation}\label{dis-Euler}
c_{Euler}(k) = \pm \sqrt{\frac{\tanh(k)}{k}}.
\end{equation}
For waves moving to the right, the $+$--sign is appropriate.    One recognizes 
that the Taylor expansion of the function of the right-hand side of \eqref{dis-Euler}
in the long wave regime (small wavenumber $k$) is 
$$
c_{Euler}(k) \, = \, 1 - \frac16 k^2 + \frac{19}{360} k^4 - \frac{55}{3024} k^6 + O(k^8).
$$
In consequence, all the models put forward here are seen to satisfy the full, 
linear dispersion relation through order $k^4$.   
Of course, 
if the derivation is done correctly, this has to be the case.  If one 
rescales the variables so the long wavelength assumption is measured by $\beta$ 
as in the formalities of the derivation, then one sees that the error in the linear 
part of the approximation is at least of order $\beta^3$.

It is tempting to choose the parameters $\theta, \lambda, \mu, \lambda_1, \mu_1$ 
and $\rho$ so that $\mathcal F$ matches 
 the next order in the dispersion relation  exactly, as was done at 
the lower order in \cite{minbona}. 
  Hence, if the auxiliary parameters 
are chosen so that
 \begin{equation}  \label{dispersion}
\mathcal{F}(\theta,\lambda,\mu,\lambda_1,\mu_1,\rho) = -\frac{55}{3024},
\end{equation} 
then the linear dispersion in the model would match that of the linear water wave 
problem up to and including  order $\beta^3$.   Such a choice could have a
salutory effect on the detailed accuracy of the model, though it  does not 
improve the overall formal level of approximation.  

Of course, one needs that the criteria for local well posedness continue to 
hold in the light of this choice.  A study of the formula \eqref{F} for $\mathcal F$
shows that 
\begin{equation}  \label{F1}
\begin{split}
\mathcal F \, &= \,  -\gamma_1 \delta_2 -\gamma_2\big(-\delta_1 + \gamma_1^2\big) + 2 \gamma_1\delta_1 - \gamma_1^3 \\
&= \,  -\gamma_1 \delta_2 +  \delta_1\big( \gamma_2 + 2 \gamma_1 \big) -\gamma_1^2\big(\gamma_1 + \gamma_2 \big)  \\
& = \, -\gamma_1\delta_2 + \delta_1\left(\gamma_1 + \frac16\right) -
\frac16 \gamma_1^2 \\
& = \, \gamma_1\left(\delta_1 - \delta_2 - \frac16\gamma_1 \right) + \frac16 \delta_1  \\
&= \, -\frac{19}{360} \gamma_1 + \frac16 \delta_1,
\end{split}
\end{equation}
where the facts that $\gamma_1 + \gamma_2 = 1/6$ and the relation 
\eqref{coefficient} have been used.   It is interesting to know whether or not
 the relation 
\eqref{dispersion}, which implies the model dispersion relation  agrees with
the exact linear dispersion relation up to order $k^6$, is consistent with the conditions  
 $\delta_1 > 0, \gamma_1 > 0$ and $\gamma = \frac{7}{48}$ implying 
global well posedness.  The condition $\gamma =  \frac{7}{48}$ requires
 that $\rho = b + d - \frac16$ as in \eqref{rho-2}.  This in turns implies 
that $\gamma_1 = \frac{1}{12} > 0$.  
That the parameters can be chosen so that \eqref{dispersion} holds is clear 
upon consulting the formula \eqref{delta} for $\delta_1$, which already 
presumes that $\rho = b + d - \frac16$.   For example, 
choose $\theta^2 \in (\frac15,1)$, and fix $\lambda, \mu$ and $\mu_1$.  
Then $\delta_1$ is seen to have the form 
$$
\delta_1 \, = \, M + N\lambda_1
$$ 
where $N > 0$.   Clearly any value of  $\delta_1$ can be achieved by
a suitable choice of $\lambda_1$ and so any value of $\mathcal F$ 
can be achieved under the restriction $\rho = b + d - \frac16$.  
However, notice that  \eqref{F1} and \eqref{dispersion} 
yield 
\begin{equation}\label{est-d1}
\delta_1 \, = \, 6\left(\frac{19}{360}\frac{1}{12} - \frac{55}{3024}    \right)
 \, = \,  -\frac{139}{1680} \, < \, 0.
\end{equation}
Hence, the requirement of Hamiltonian structure together with  local well-posedness 
are not consistent with the model approximating the dispersion relation 
at the next order without considering  $O(\alpha^2, \beta^2, \alpha\beta)$ terms in \eqref{kdvbbm-2} and a new correction parameter like $\rho$.

\section{Concluding Remarks}  \label{sec-6}

Derived here is a class of unidirectional models for long-crested water waves that are formally second-order correct.
Basic analysis of the pure initial-value problem for our models has been developed.  A local well-posedness 
theory in relatively weak spaces is established under  conditions on the two 
parameters $\delta_1$ and $\gamma_1$
that appear in the model, and which depend upon the other parameters.  
Global well-posedness is only established in case the equation has a special, Hamiltonian
structure.  Conditions under which both aspects obtain are given.  

A comment is deserved about the focus maintained throughout
 on unidirectional models.  
Boussinesq himself 
understood that his one-way model was simpler than the coupled pair of two-way 
models that he first derived.     It was also simpler than a second-order in time,
unidirectional model  
equation he had derived earlier.  In both these instances, a
 modern perspective on this issue is that the 
undirectional model 
can be posed with half the auxiliary data needed to initiate the coupled system.  
However,  unidirectionality places a severe limitation on 
the wave motion when it is posed as an initial-value problem.  More precisely, 
a strict relationship between the initial wave profile and the velocity field is 
implied.  On the other hand,
it is known that for Boussinesq-type systems, if the initial disturbance is suitably 
localized and small, then on certain temporal scales, the disturbance will decompose into 
a left- and a right-going wave, each of which satisfy approximately a unidirectional 
equation  (see \cite{SchneiderWayne}, \cite{BCL1}).   Finally,  it is worth noting that even  fairly steep beaches do not reflect all that much energy (see \cite{MP}). 
 For very gently shelving beaches such as obtain in many nearshore zones,
 the reflection is negligible as regards its effect on shaping and erosive 
processes. 
Hence, unidirectional models seem to suffice in such circumstances.    

  Finallly, we remark that when 
 choosing  the depth parameter $\theta$, it is a good idea if it is taken well inside the interval $[0,1]$.   While the horizontal velocity does not appear in the unidirectional 
model, a formal corollary of its derivation is a prediction of the horizontal velocity at the depth $1-\theta^2$.  This is comprised of the formula (\ref{eq1.3}) 
expressing the horizontal velocity in terms of the functions $A, B, C, D$ and $E$
together with the forms (\ref{first-orderAB}) determined for $A$ and $B$ and those  for $C, D$ and $E$.   It is hard to measure 
the horizontal velocity very close to the free surface, 
while in actual fact, there is no velocity on the bottom because of the viscous boundary
layer.  Typical velocity measurements in laboratory and field situations are made somewhere in the middle of the water column.


\section*{Acknowledgments}
The authors are thankful to an anonymous referee who pointed out an error 
in the original version of the manscript and whose comments helped to improve the manuscript.  JB  thanks the Instituto de Matem\'atica, Estat\'{\i}stica e Computa\c c\~ao Cient\'{\i}fica  (IMECC) at the State University of Campinas, S\~ao Paulo and the University of Illinois 
at Chicago 
for support.  He also held a FAPESP 2016/01544-2 grant at IMECC during the final stages of the work.  Part of the work was done while visiting 
 the National Center for 
Theoretical Sciences (NCTS) at Taiwan National University. He is grateful for the
excellent support and fine working conditions.   
XC also thanks IMECC for support as a longer-term visitor
under the project FAPESP 2012/23054-6, and support from CNPq grants
304036/2014-5 and 481715/2012-6.  MP acknowledges support from FAPESP 2012/20966-4, 
FAEPEX 1486/12 and CNPq 479558/2013-2 \& 305483/2014-5.  MP and MS thank the Department of Mathematics, 
Statistics and Computer Science at the University of Illinois at Chicago for hospitality and
 support during part of this collaboration.


\end{document}